\documentclass[12pt]{amsart}
\usepackage{amsmath,amsthm,amssymb,amsfonts,amscd}
\usepackage{epsfig}
\usepackage{xypic}
\usepackage[colorlinks, citecolor=blue]{hyperref}
\numberwithin{equation}{section}
\theoremstyle{plain}
\newtheorem{theorem}{Theorem}
\newtheorem{conjecture}[theorem]{Conjecture}
\newtheorem{lemma}[theorem]{Lemma}
\newtheorem{corollary}[theorem]{Corollary}
\newtheorem{proposition}[theorem]{Proposition}
\newtheorem*{theorem*}{Theorem}
\newtheorem*{conjecture*}{Conjecture}

\theoremstyle{definition}
\newtheorem{remark}[theorem]{Remark}
\newtheorem{example}[theorem]{Example}
\newtheorem*{definition}{Definition}
\newtheorem{problem}[theorem]{Problem}
\newcommand{\CC}{{\mathbb C}}

\newcommand{\FF}{{\mathbb F}}
\newcommand{\RR}{{\mathbb R}}
\newcommand{\ZZ}{{\mathbb Z}}

\newcommand{\NN}{{\mathbb N}}

\newcommand{\calO}{{\mathcal O}}

\newcommand{\calP}{{\mathcal P}}

\newcommand{\calB}{{\mathcal B}}
\newcommand{\calX}{{\mathcal X}}
\newcommand{\calY}{{\mathcal Y}}
\newcommand{\calD}{{\mathcal D}}
\newcommand{\calS}{{\mathcal S}}

\newcommand{\eps}{\varepsilon}

\newcommand{\funktion}[5]{\begin{array}{rccl} {#1}: & {#2} &
\longrightarrow & {#3}\\ & {#4} & \longmapsto & {#5} \end{array}}
\newcommand{\on}[1]{\operatorname{#1}}

\title[Distinguished bases]{Distinguished bases and monodromy of complex hypersurface singularities}
\author{Wolfgang Ebeling}
\thanks{Partially supported by DFG}
\address{Institut f\"ur Algebraische Geometrie, Leibniz Universit\"at Hannover, Postfach 6009, D-30060 Hannover, Germany}
\email{ebeling@math.uni-hannover.de}

\date{}

\makeindex

\begin{document}
%\selectlanguage{english}

\maketitle

%\begin{center}
%{\it Dedicated to the memory of Egbert Brieskorn with great admiration}
%\end{center}

\begin{abstract} 
We give a survey on some aspects of the topological investigation of isolated singularities of complex hypersurfaces by means of Picard-Lefschetz theory.  We focus on the concept of distinguished bases of vanishing cycles and the concept of monodromy. 
\end{abstract}

\tableofcontents

%%%%%%%%%%%%%%%%%%%%%%%%
\section*{Introduction}
The pioneering fibration theorem of J.~Milnor \cite{Mi68} opened the way to study the topology of isolated complex hypersurface singularities. To study the topology of real smooth manifolds one can use Morse theory. The idea of Morse theory is that the topological type of the level set of a real function changes when passing through a critical value. In order to study the topology of the singularity defined by a complex analytic function one can investigate the level sets of this function. The complex analogue of Morse theory is Picard-Lefschetz theory. It is older than Morse theory and goes back to E.~Picard and S.~Simart \cite{PS97} and to S.~Lefschetz \cite{Lef24}. 

Around 1967--1969, the Picard-Lefschetz theory experienced a revival when it was brought into an algebraic form by A.~Grothendieck, P.~Deligne, and N.~Katz in \cite{DK73}. On a more modest scale,
the theory was applied in the late 1960s and early 1970s to the analysis of isolated singularities of complex hypersurfaces. The first fundamental contributions were made by F.~Pham \cite{Ph65}, L\^e D\~ung Tr\'ang \cite{Le73, Le78}, E.~Brieskorn \cite[Appendix]{Br70}, K.~Lamotke \cite{Lam75}, and A.~M.~Gabrielov \cite{Ga73, Ga74a, Ga74b, Ga79}. Gabrielov coined the notion of ``distinguished bases''. Instead of passing through a critical value, the fundamental principle of Picard-Lefschetz theory is going around a critical value in the complex plane. Roughly speaking, to the critical values there corresponds a distinguished basis of vanishing cycles and the change of the topology of the level set is given by the ``monodromy''. This article is a survey of these fundamental concepts and the further developments.

Nowadays, there are good references for this subject. There is a survey article by S.~M.~Gusein-Zade \cite{GZ77} and a later one by Brieskorn  \cite{Br88}. A very good reference is the second volume of the book of V.~I.~Arnold, Gusein-Zade, and A.~Varchenko \cite{AGV88}. The book of E.~Looijenga \cite{Loo84} is devoted to isolated complete intersection singularities, but it also contains relevant information about hypersurface singularities which are a special case. Moreover, there are also textbooks by D.~B\"attig and H.~Kn\"orrer \cite{BK91} (in German) and by the author \cite{Eb07}. The author has already written a survey on the classical monodromy \cite{Eb06}. We keep the intersection with this survey to a minimum. We give almost no proofs, but provide precise references to these books as well as to the original articles for details, including proofs. 

Let me outline the contents of this article. In the first section, we introduce the notion of a distinguished basis of vanishing cycles. More precisely, we define distinguished and weakly distinguished bases. In the second section, we consider the intersection form, the classical monodromy, and the Seifert form and we show how matrices of these invariants with respect to distinguished bases are related to one another. Moreover, we define the concept of Coxeter-Dynkin diagram. In Sect.~\ref{sect:braid}, we consider the change of basis and introduce the action of the braid group on the set of distinguished bases. In Sect.~\ref{sect:comp}, we collect together results about the computation of intersection matrices and Seifert matrices with respect to distinguished bases. In Sect.~\ref{sect:LL}, we discuss the implication of the irreducibility of the discriminant to properties of the invariants and we introduce the Lyashko-Looijenga map. In Sect.~\ref{sect:special}, we review Arnold's classification of singularities and compile explicit results for the simple, unimodal, and bimodal singularities. Sect.~\ref{sect:mono} is devoted to an algebraic description of the monodromy group. Finally, in Sect.~\ref{sect:top}, we consider the question to which extent the invariants determine the topological type of the singularities. We conclude with some open problems.

The notion of distinguished bases can also be generalized to isolated complete intersection singularities,  see \cite{Eb87}. We shall not discuss this case in this survey, we restrict ourselves to isolated complex hypersurface singularities. 

There are many further generalizations and applications of the theory, even outside of singularity theory. We mention some of the results, but mainly indicate references to the corresponding articles. We do not claim to be complete. 

\noindent {\bf Acknowledgements.} The author is grateful to C.~Hertling for drawing his attention to an error in his book \cite{Eb07} (see the proof of Theorem~\ref{thm:SLH}) and for useful discussions.

%%%%%%%%%%%%%%%%%%%%%%%%%%%%%%%%%%%%
\section{Distinguished bases of vanishing cycles} \label{sect:dist}

Let $f:(\CC^{n+1},0) \to (\CC,0)$ be the germ of a holomorphic function with an isolated singularity at the
origin. This means that 
\[ {\rm grad} f(a)= \left(\frac{\partial f}{\partial z_0}(a), \ldots , \frac{\partial f}{\partial z_{n}}(a) \right)  \neq 0
\]
for all points $a \neq 0$ in a small neighborhood of the origin, $(z_0, \ldots, z_{n})$ denote the coordinates of $\CC^{n+1}$. For short, we call $f$ a singularity.

One has the famous result of Milnor \cite{Mi68}:
Let $\eps
>0$ be small enough such that the closed ball $B_\eps \subset \CC^{n+1}$ of
radius $\eps$ around the origin in $\CC^{n+1}$ intersects the fiber $f^{-1}(0)$
transversely. Let $0 < \eta \ll \eps$ be such that for $t$ in the closed disc
$\Delta \subset \CC$ of radius $\eta$ around the origin, the fiber $f^{-1}(t)$
intersects the ball $B_\eps$ transversely. Let 
\begin{eqnarray*} 
X_t & := & f^{-1}(t) \cap B_\eps  \mbox{ for } t \in \Delta,\\ 
X & := & f^{-1}(\Delta) \cap B_\eps, \\ 
X^\ast& := & X \setminus X_0,\\ 
\Delta^\ast & := & \Delta \setminus \{ 0\}.
\end{eqnarray*} 
By a result of J.~Milnor \cite{Mi68}, the mapping
$f|_{X^\ast} : X^\ast \to \Delta^\ast$ is the projection of a (locally trivial) ($C^\infty$)-differentiable fiber bundle.
The fiber $X_\eta$ over the point $\eta \in \Delta^\ast$ is a $2n$-dimensional differentiable manifold with
boundary which has the homotopy type of a bouquet of $\mu$ $n$-spheres where $\mu$ is the Milnor number of the
singularity. This differentiable fiber bundle $(X^\ast,f|_{X^\ast},\Delta^\ast,X_\eta)$ is called the {\em
Milnor fibration}\index{Milnor fibration}\index{fibration!Milnor} and the typical fiber $X_\eta$ is called the {\em Milnor fiber}\index{Milnor fiber}\index{fiber!Milnor}. The only non-trivial reduced homology group is the group $\widetilde{H}_n(X_\eta; \ZZ)$. It is equipped with the intersection form $\langle \ , \ \rangle$. This bilinear form is symmetric if $n$ is even and skew-symmetric if $n$ is odd. We shall only consider homology with integral coefficients and we shall write $\widetilde{H}_n(X_\eta)$ for $\widetilde{H}_n(X_\eta; \ZZ)$ in the sequel.

\begin{definition} The group $\widetilde{H}_n(X_\eta)$ together with the intersection form $\langle \ , \ \rangle$ is called the {\em Milnor lattice}\index{Milnor lattice}\index{lattice!Milnor} of $f$ and denoted by $M$.
\end{definition}

The Milnor lattice $M$ is a lattice, i.e., a free $\ZZ$-module of finite rank equipped with a symmetric or skew-symmetric bilinear form $\langle \ , \ \rangle$. The rank of the Milnor lattice is the Milnor number $\mu$.

Let  $\omega$ be the loop
\[ \begin{array}{lccc} \omega : & [0,1] & \to & \CC \\
                                                & t & \mapsto & \eta e^{2 \pi \sqrt{-1} t} .
                                                \end{array}
\]
Then parallel translation along this path induces a diffeomorphism $h=h_\omega: X_\eta \to X_\eta$ which is called the {\em geometric monodromy}\index{geometric monodromy}\index{monodromy!geometric} of the
singularity $f$. 

\begin{definition}
The induced homomorphism $h_\ast :M\to M$ on the Milnor lattice $M$ is called the {\em (classical)
monodromy}\index{classical monodromy}\index{monodromy!classical} (or the {\em (classical)
monodromy operator})\index{classical monodromy operator}\index{operator!classical monodromy} of the singularity $f$.
\end{definition}

Our aim is to study the Milnor fibration, the Milnor lattice $M$, and the monodromy.

For this purpose, we shall consider a morsification of the function $f$. This is defined as follows. An {\em unfolding}\index{unfolding} of $f$ is a holomorphic function germ $F:(\CC^{n+1} \times \CC^k,0) \to (\CC,0)$ with $F(z,0)=f(z)$ (see \cite[1.2]{Gr19}). A {\em morsification}\index{morsification} is a representative $F: V \times U \to \CC$ of an unfolding
\[ \begin{array}{lccc} F : & (\CC^{n+1} \times \CC,0) & \to & (\CC,0) \\
                                      & (z,\lambda) & \mapsto & f_\lambda(z) 
                                      \end{array}
\]   
of $f$ such that for almost all $\lambda \in U \setminus \{ 0\}$ (everywhere except from a Lebesgue null set) the function $f_\lambda : V \to \CC$ is a Morse function, i.e., has only non-degenerate critical points with distinct critical values. The Morse function $f_\lambda$ is itself often called a morsification of $f$. One can show that $f$ has a morsification (see, e.g., \cite[Proposition~3.18]{Eb07}).                                   

Let $\lambda$ be chosen so that $f_\lambda$ is a Morse function. Let $Y:=f_\lambda^{-1}(\Delta) \cap B_\eps$ and $Y_t:=f_\lambda^{-1}(t) \cap B_\eps$
for $t \in \Delta$. Assume that $\lambda \neq 0$ is chosen so small that all the critical points are contained
in the interior of $Y$ and the fiber
$f_\lambda^{-1}(t)$ for $t \in \Delta$ intersects the ball
$B_\eps$ transversely. Denote the critical points by $p_1,\ldots,p_\mu$ and the critical values by
$s_1,\ldots,s_\mu$.  Assume
that $\eta \in \partial\Delta$ is a non-critical value of $f_\lambda$. 
Let $\Delta':=\Delta \setminus \{s_1,\ldots,s_\mu\}$ and $Y':=Y \cap f_\lambda^{-1}(\Delta')$.  Then the mapping
$f_\lambda|_{Y'} : Y' \to \Delta'$ is the 
projection of a differentiable fiber bundle. The fiber $Y_t$ for $t \in \Delta' \cap \Delta^\ast$ is
diffeomorphic to $X_t$. In particular, $Y_\eta$ is diffeomorphic to $X_\eta$. We therefore identify these fibers.

Let $\gamma:[0,1]\to\Delta$ be a piecewise differentiable path which connects the critical value
$s_i$ with $\eta$ and does not pass through any other critical value,
i.e.\ $\gamma(0)=s_i$, $\gamma(1)=\eta$ and $\gamma((0,1])\subset\Delta'$. 
By the complex Morse lemma there exists a neighborhood $B_i$ of the non-degenerate
critical point $p_i$ over $s_i$  and local coordinates $(z_0,\ldots,z_{n})$ centered at the point $p_i$ such that
$f_\lambda$ can be written in $B_i$ in the form
$$
f_\lambda(z_0,\ldots,z_{n})=s_i+z_0^2+\ldots+z_{n}^2
$$
and $B_i$ is a ball of radius
$\eps$ centered at $0$ in these coordinates. For sufficiently small $t>0$ the fiber ${
X}_{\gamma(t)}$ contains an $n$-sphere
$$
S(t):=\sqrt{\gamma(t)-s_i}\, S^n
$$
where
$S^n$ is the $n$-dimensional unit sphere
\[
S^n = \{ (z_0, \ldots  , z_n) \in \CC^{n+1} \, | \, {\rm Im}\, z_i=0, \ \sum z_i^2 =1 \}.
\] 
By parallel translation along $\gamma$ one obtains an $n$-sphere
$S(t)\subset { X}_{\gamma(t)}$ for each $t\in (0,1]$. For $t=0$ the sphere
$S(t)$ shrinks to the critical point $p_i$
(cf.\ Fig.~\ref{abbIII12}).
\begin{figure}
$$
\unitlength1cm
\begin{picture}(6,5.2)
\put(0.5,0.5){\includegraphics{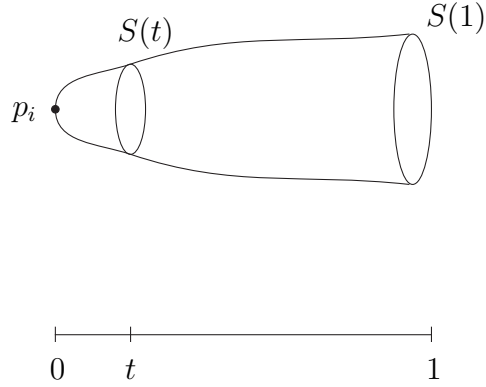}}
\put(0.5,0){$0$} \put(1.5,0){$t$} \put(5.5,0){$1$}
\put(0,3.5){$p_i$} \put(1.4,4.5){$S(t)$} \put(5.5,4.7){$S(1)$}
\end{picture}
$$
\caption{Vanishing cycle}\label{abbIII12}
\end{figure}
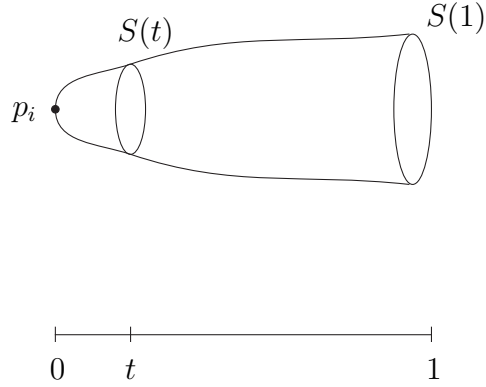
We now choose an orientation of $S(1)$. Then $S(1)$ is an $n$-cycle and 
represents a homology class $\delta$ in the Milnor lattice $M=\widetilde{H}_n(X_\eta)$.

\begin{definition}
The homology class $\delta\in M$ is called a {\em vanishing cycle}\index{vanishing cycle}\index{cycle!vanishing} of 
$f_\lambda$ (along $\gamma$).
Denote by  $\Lambda^* \subset M$\index{$\Lambda^*$} the set of vanishing cycles of $f$ (for all possible choices of a morsification,  a critical point, a path $\gamma$, and an orientation).
 \end{definition}

 A vanishing cycle is well defined up to orientation.

For the self-intersection number of the
vanishing cycle $\delta$ in the Milnor fiber $X_\eta$  one has the following result (see also \cite[Lemma~1.4]{AGV88}, \cite[Proposition~5.3]{Eb07}).
\begin{proposition} \label{prop:self}
The vanishing cycle $\delta \in M$ has the self-intersect\-ion number
$$
\langle \delta , \delta \rangle = (-1)^{n(n-1)/2}(1+(-1)^n)= \left\{ \begin{array}{cl} 
0 & \mbox{for $n$ odd,} \\
2 & \mbox{for $n \equiv 0$ (mod $4$),} \\
-2 &  \mbox{for $n \equiv 2$ (mod $4$).} 
\end{array} \right.
$$
\end{proposition}

\begin{proof}
In order to compute the self-intersection number
$\langle\delta,\delta\rangle$ of the vanishing cycle $\delta$, it suffices to compute the self-intersection number of the sphere $S^n$ in the complex manifold
\[ Z = \{ (z_0, \ldots , z_n)\in \CC^{n+1} \, | \, z_0^2+ \cdots + z_n^2=1 \}.
\]
It is easy to see that the
manifold  $Z$ is diffeomorphic to the total space $TS^n$ of the tangent bundle
of the sphere $S^n$ which can be described as follows:
\[
TS^n= \left\{ u + \sqrt{-1}v \in \CC^{n+1} \, \left| \, \sum u_i^2=1, \sum u_i v_i=0 \right\} \right. .
\]
A diffeomorphism from the manifold
$Z$ to $TS^n$ can be defined by 
\[ z_i = x_i + \sqrt{-1} y_i \mapsto u_i+ \sqrt{-1} v_i = \frac{x_i}{|x|} + \sqrt{-1}  y_i
\]
where  $|x|=\sqrt{\sum x_i^2}$. This diffeomorphism sends
the unit sphere $S^n\subset Z$ to the zero section of
the tangent bundle $TS^n$. The self-intersection
number of the zero section $S^n$ in the total space of the tangent bundle $TS^n$  is equal to the Euler characteristic 
$\chi(S^n)=1+(-1)^n$.
However,  the natural
orientations of the manifolds $Z$ (as a complex analytic
manifold) and $TS^n$ (as the total space of a tangent bundle) differ by the sign $(-1)^{n(n-1)/2}$. 
\end{proof}

The path $\gamma:I\to\Delta$ defines a closed path around the critical value $s_i$ in the following way:  Let $\Delta_i$
be a disc of sufficiently small radius $\eta_i$ around $s_i$ such that $\gamma(I)$ intersects the boundary 
$\partial {\Delta}_i$ of ${\Delta}_i$ exactly once, namely at time
$t=\theta$ at the point $s_i+u_i$. Let $\tau:I\to {\Delta}_i$, $t\mapsto
s_i+u_ie^{2\pi\sqrt{-1}t}$, be the path starting at $s_i+u_i$ which goes once around $s_i$ on the boundary of 
${\Delta}_i$ in counterclockwise direction. Moreover, set
$\widetilde\gamma:=\gamma|_{[\theta,1]}$.
The closed path $\omega={\widetilde\gamma}^{-1}\tau\widetilde\gamma$ with starting and end point
$\eta$ is called the {\em simple loop}\index{simple loop}\index{loop!simple} associated to
$\gamma$ (cf.\ Fig.~\ref{abbIII13}).
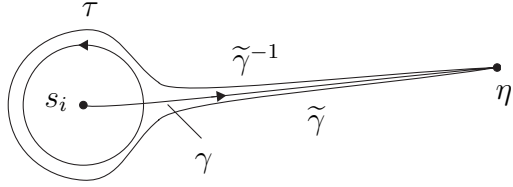
\begin{figure}
$$
\unitlength1cm
\begin{picture}(6.6,3)
\put(0,0){\includegraphics{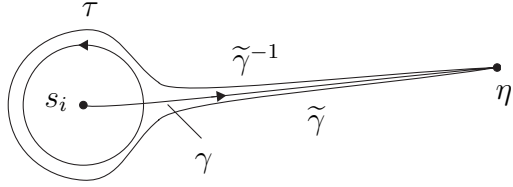}}
\put(0.5,1){$s_i$} \put(1,2.2){$\tau$} \put(2.5,0.2){$\gamma$}
\put(4,0.7){$\widetilde\gamma$} \put(3,1.5){${\widetilde\gamma}^{-1}$}
\put(6.5,1.1){$\eta$}
\end{picture}
$$
\caption{Simple loop associated to $\gamma$}\label{abbIII13}
\end{figure}
The monodromy 
$$
h_\delta:=h_{\omega *}:M\longrightarrow M
$$
corresponding to the simple loop $\omega$ associated to $\gamma$  is called the {\em
Picard-Lefschetz transfor\-ma\-tion}\index{Picard-Lefschetz transformation}\index{transformation!Picard-Lefschetz}
corresponding to the vanishing cycle $\delta$.

The following theorem is the basic result of the Picard-Lefschetz theory. It goes back to Picard and Simart \cite[p.~95ff.]{PS97} and Lefschetz \cite[Th\'eor\`eme fondamental, p.~23 \& p.~92]{Lef24}. For a proof see \cite[\S 5]{Lam75}, \cite[Chapter~3]{Loo84}, and \cite[1.3]{AGV88}. A proof following the proof in Looijenga's book \cite[Chapter~3]{Loo84} is also given in \cite[\S 5.3]{Eb07}. For a modern account of Picard-Lefschetz theory see also the article of Lamotke \cite{Lam81}.

\begin{theorem}[Picard-Lefschetz formula] \label{thm:P-L}
For $\alpha \in M$ we have
$$
h_\delta(\alpha)=\alpha-(-1)^{\frac{n(n-1)}{2}}\langle\alpha,\delta\rangle
\delta.
$$
\end{theorem}

When $n$ is even, the intersection form $\langle \ , \ \rangle$ is a symmetric bilinear form and we can combine the formulas from Proposition~\ref{prop:self} and Theorem~\ref{thm:P-L} together as
\[
h_\delta(\alpha) = \alpha - \frac{2 \langle \alpha, \delta \rangle}{\langle \delta, \delta \rangle} \delta.
\]
This means that the operator $h_\delta : M \to M$ is a {\em reflection}\index{reflection} in the hyperplane of $M$ orthogonal to $\delta$. Such a reflection is also denoted by $s_\delta$, so in this case $h_\delta=s_\delta$. When $n$ is odd, the intersection form $\langle \ , \ \rangle$  is skew symmetric and Theorem~\ref{thm:P-L} means that $h_\delta$ is a {\em symplectic transvection}\index{symplectic transvection}\index{transvection!symplectic}.

We now assume that $\eps$ and $\eta$ are chosen so small that all the balls ${ B}_i$ 
and all the discs ${\Delta}_i$ are disjoint. We consider an ordered system
 $(\gamma_1,\ldots,\gamma_\mu)$ of paths
$\gamma_i:I\to\Delta$ with $\gamma_i(0)=s_i$, $\gamma_i(1)=\eta$ and
$\gamma_i((0,1])\subset\Delta'$.

\begin{definition}
The system $(\gamma_1,\ldots,\gamma_\mu)$ of paths is called {\em
 distinguished}\index{distinguished system}\index{system!distinguished} if the following conditions are satisfied:
\begin{itemize}
\item[(i)] The paths $\gamma_i$ are non-selfintersecting.
\item[(ii)] The only common point of $\gamma_i$ and $\gamma_j$ for
$i\neq j$ is $\eta$.
\item[(iii)] The paths are numbered in the order in which they arrive at $\eta$ where one has to count clockwise
from the boundary of the disc (cf.\ Fig.~\ref{abbIII14}).
\end{itemize}

A system $(\delta_1,\ldots,\delta_\mu)$ of vanishing cycles $\delta_i\in \Lambda^*$ is called {\em distinguished}, if there exists a  distinguished system 
$(\gamma_1,\ldots,\gamma_\mu)$ of paths such that
$\delta_i$ is a cycle vanishing along $\gamma_i$.
\end{definition}

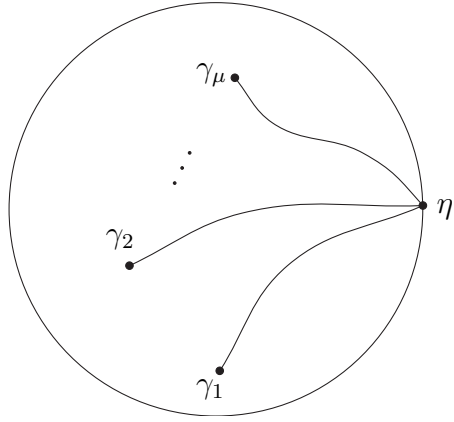
\begin{figure}
$$
\unitlength1cm
\begin{picture}(5.5,5.5)
\put(0,0){\includegraphics{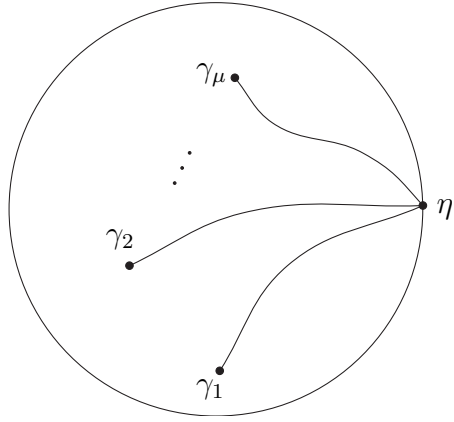}}
\put(5.7,2.7){$\eta$} \put(2.5,4.5){$\gamma_\mu$} \put(1.3,2.3){$\gamma_2$}
\put(2.5,0.3){$\gamma_1$}
\end{picture}
$$
\caption{Distinguished system of paths}\label{abbIII14}
\end{figure}

Since $\Delta'$ is a disc from which $\mu$ points have been deleted, its fundamental group $\pi_1(\Delta',\eta)$
 is the free group on $\mu$ generators.
 If $(\gamma_1,\ldots,\gamma_\mu)$ is a distinguished system of
 paths,
 then $\pi_1(\Delta',\eta)$ is the free group on the generators $\omega_1,\ldots,\omega_\mu$, where $\omega_i$ is the
 simple loop associated to $\gamma_i$.

 \begin{definition}
 The system $(\gamma_1,\ldots,\gamma_\mu)$ of paths is called
 {\em weakly distinguished}
 \index{weakly distinguished system}\index{system!weakly distinguished}
 if~$\pi_1(\Delta',\eta)$ is the free group on the generators
 $(\omega_1,\ldots,\omega_\mu)$,
 where $\omega_i$ is the simple loop belonging to $\gamma_i$.

 A system $(\delta_1,\ldots,\delta_\mu)$ of vanishing cycles $\delta_i\in \Lambda^*$
 is called {\em weakly distinguished}
 if $\delta_i$ is a vanishing cycle along a path $\gamma_i$ of a weakly distinguished system
 $(\gamma_1,\ldots,\gamma_\mu)$ of paths.
 \end{definition}

 Note that the numbering is important for a distinguished system of paths, but of no significance for a weakly distinguished system of paths.
 A distinguished system of paths is of course also weakly distinguished.
 
Brieskorn proved the following theorem \cite[Appendix]{Br70} (see also \cite[Theorem~2.1]{AGV88}, \cite[Proposition~5.5]{Eb07}).

\begin{theorem}[Brieskorn] \label{thm:distbasis}
A  distinguished system $(\delta_1,\ldots,\delta_\mu)$ of vanishing 
cycles is a basis of the lattice $M$, i.e., $\langle \delta_1, \ldots, \delta_\mu \rangle_\ZZ=M$, where $\langle \delta_1, \ldots, \delta_\mu \rangle_\ZZ$ denotes the $\ZZ$-span of $(\delta_1, \ldots, \delta_\mu)$.
\end{theorem}

From this theorem, one can derive the following corollary (see \cite[Theorem~2.8]{AGV88}, \cite[Proposition~5.6]{Eb07}).

\begin{corollary} \label{cor:weakbasis}
A weakly distinguished system $(\delta_1,\ldots,\delta_\mu)$ of vanishing 
cycles also forms a basis of $M$. 
\end{corollary}

 \begin{definition}
 A basis $(\delta_1,\ldots,\delta_\mu)$ of $M$ is called
 {\em distinguished} (resp.\ {\em weakly distinguished}) 
 \index{distinguished basis}\index{basis!distinguished}
 \index{weakly distinguished basis}\index{basis!weakly distinguished}
 if  $(\delta_1,\ldots,\delta_\mu)$ is a distinguished (resp.~weakly distinguished) system of vanishing cycles.
 \end{definition}

 By Theorem~\ref{thm:distbasis} and Corollary~\ref{cor:weakbasis} every distinguished or weakly distinguished system of vanishing cycles forms a
 basis.

 The concepts \lq\lq distinguished\rq\rq\ and \lq\lq weakly distinguished\rq\rq\
 are due to Gabrielov.
 In order to distinguish both concepts better, one sometimes says, following a suggestion of Brieskorn,
 \lq\lq strongly distinguished\rq\rq\ instead of \lq\lq distinguished\rq\rq.
 The term \lq\lq geometric basis\rq\rq\index{geometric basis}\index{basis!geometric} is also used for a distinguished
 basis.
 
The group of all automorphisms of a lattice $M$, i.e., isomorphisms $M \to M$ which respect the bilinear form, will be denoted by ${\rm Aut}(M)$.

 \begin{definition}
 The image $\Gamma$ of the homomorphism
  $$
  \funktion{\rho}{\pi_1(\Delta',\eta)}{\on{Aut}(M)}{[\gamma]}{h_{\gamma
  *}}
  $$
 is called the {\em monodromy group}\index{monodromy group}\index{group!monodromy} of the singularity $f$.
\end{definition}

If $(\delta_1, \ldots , \delta_\mu)$ is a weakly distinguished basis, then the monodromy group of $f$ is generated by the Picard-Lefschetz transformations $h_{\delta_i}$ corresponding to the vanishing cycles $\delta_i$. Therefore the monodromy group of $f$ is a group with $\mu$ generators. Indeed, the monodromy group is independent of the morsification of $f$, see Theorem~\ref{thm:bif} below. 

\begin{example} \label{ex:A_k}
(For this example see also \cite[2.9]{AGV88} and \cite[Example~5.4]{Eb07}.)
We consider the function $f: \CC \to \CC$ with $f(z)=z^{k+1}$. (This is the singularity $A_k$, see Sect.~\ref{sect:special}.) The Milnor fiber $X_\eta$ consists of $k+1$ points, namely the $(k+1)$-th roots of $\eta$. As a morsification of $f$ we consider the function $f_\lambda(z)=z^{k+1}-\lambda z$ for  $\lambda \neq 0$. Fix $\lambda \in \RR$, $\lambda > 0$. The critical points of the function $f_\lambda$ are given by the equation
\[ f'_\lambda(z)= (k+1)z^k - \lambda =0. 
\]
Therefore they are the points
  $$
  p_i=\sqrt[k]{\frac{\lambda}{k+1}}\xi_i,\quad \xi_i=e^{-\frac{2\pi
  i\sqrt{-1}}{k}},
  $$
 with the critical values
  $$
  s_i=-\frac{\lambda k}{k+1}\sqrt[k]{\frac{\lambda}{k+1}}\xi_i,\quad
  i=1,\ldots,k.
  $$

 As a noncritical value we choose $-\eta$, where $\eta\in\RR$, $\eta>0$ and
  $$ \eta \gg \frac{\lambda k}{k+1}\sqrt[k]{\frac{\lambda}{k+1}}.
  $$
Let $\gamma_i:[0,1]\to\bar\Delta$, $t\mapsto (1-t)s_i$,
 and let $\tau$ be a path from  $0$ to $-\eta$ which runs along the real axis and goes once around the critical value
  $$
  s_k=-\frac{\lambda k}{k+1}\sqrt[k]{\frac{\lambda}{k+1}}\xi_k\in\RR
  $$
 in the positive direction.

We consider the path system $(\gamma_1\tau,\ldots,\gamma_k\tau)$.
This system is homotopic to a  distinguished path system. (For the notion of homotopy of path systems see Sect.~\ref{sect:braid} below.)
 Let $(\delta_1,\ldots,\delta_k)$ be a corresponding  distinguished system of vanishing cycles
 in ${\widetilde H}_0(X_{-\eta})$. 
 
 In order to compute the intersection numbers $\langle\delta_i,\delta_j\rangle$ of the vanishing cycles
 in ${\widetilde H}_0(X_{-\eta})$ we transport the system $(\delta_1,\ldots,\delta_k)$
 by parallel transport along the path $\tau^{-1}$ to ${\widetilde H}_0(X_{0})$.
 We thus consider a system of vanishing cycles in ${\widetilde H}_0(X_{0})$, which we again denote by
 $(\delta_1,\ldots,\delta_k)$, and which is defined by the path system $(\gamma_1,\ldots,\gamma_k)$.

 The fiber $X_0$ consists of the $k+1$ points
  $$
  x_0=0,x_1=\sqrt[k]{\lambda}\xi_1,\ldots,x_k=\sqrt[k]{\lambda}\xi_k.
  $$
 Then up to orientation  $\delta_i$ is represented by the cycle $x_i-x_0$. It is easy to calculate that $x_i-x_0$ vanishes along
 $\gamma_i$, i.e.,  that the points $x_i$ and $x_0$ fall together along
 $\gamma_i$. Let
  $$
  \delta_i=[x_i-x_0],\ i=1,\ldots,k.
  $$
 Then
  $$
  \langle\delta_i,\delta_j\rangle=\left\{\begin{array}{cl} 2 & \text{ for }i=j,\\
  1 & \text{ for }i\neq j. \end{array}\right.
  $$
In this case, the Milnor lattice $M$, the set of vanishing cycles $\Lambda^*$, and the monodromy group $\Gamma$ can be described as follows. Let $e_1, \ldots , e_{k+1}$ be the standard basis of $\RR^{k+1}$ and $\langle \ , \ \rangle$ the Euclidean scalar product on $\RR^{k+1}$. Denote by $\calS_{k+1}$ the symmetric group in $k+1$ symbols. Then
\begin{eqnarray*}
 M & =& \{ (v_1, \ldots , v_{k+1}) \in \ZZ^{k+1} \, | \, v_1 + \cdots + v_{k+1} =0 \}, \\
\Lambda^\ast & = & \{ e_i - e_j \, | \, 1 \leq i,j \leq k+1, i \neq j \} = \{ v \in M \, | \, \langle v,v \rangle =2 \},\\
\Gamma & = & \calS_{k+1}.
\end{eqnarray*}
 \end{example}
 
 %%%%%%%%%%%%%%%%%%%%
\section{Coxeter-Dynkin diagram and Seifert form} \label{sect:CD}

\begin{definition}
 Let $(\delta_1,\ldots,\delta_\mu)$ be a weakly distinguished basis of $M$.
 The matrix
  $$
  S:=(\langle\delta_i,\delta_j\rangle)_{j=1,\ldots,\mu}^{i=1,\ldots,\mu}
  $$
 is called the {\em intersection matrix}\index{intersection matrix}\index{matrix!intersection} of $f$ with respect to
 $(\delta_1,\ldots,\delta_\mu)$.
 \end{definition}

 By Proposition~\ref{prop:self}, the diagonal entries of the intersection matrix
 satisfy
  $$
  \langle\delta_i,\delta_i\rangle=(-1)^{\frac{n(n-1)}{2}}(1+(-1)^n)\text{ for all }i.
  $$

 It is usual to represent the intersection matrix by a graph called the
 Coxeter-Dynkin diagram.

\begin{sloppypar}

 \begin{definition}
 Let $(\delta_1,\ldots,\delta_\mu)$ be a weakly distinguished basis of $M$.
 The {\em Coxeter-Dynkin diagram}\index{Coxeter-Dynkin diagram}\index{diagram!Coxeter-Dynkin}
 of the singularity $f$ with respect to $(\delta_1,\ldots,\delta_\mu)$ is the graph $D$ defined as follows:

 (i) The vertices of $D$ are in one-to-one correspondence with the elements
 $\delta_1,\ldots,\delta_\mu$.

 (ii) For $i<j$ with $\langle\delta_i,\delta_j\rangle\neq0$  the $i$-th and the $j$-th vertex are connected by $|\langle\delta_i,\delta_j\rangle|$
 edges, weighted with the sign $+1$ or $-1$ of
 $\langle\delta_i,\delta_j\rangle\in\ZZ$.
 We indicate the weight
  $$
  w=\left\{\begin{array}{cl} (-1)^{\frac n2} & \text{ for }n\text{ even,}\\
  (-1)^{\frac{n+1}{2}} & \text{ for }n\text{ odd}\end{array}\right.
  $$
 by a dashed line, the weight $-w$ by a solid line.
 \end{definition}
 
 \end{sloppypar}
 
 These diagrams are usually called Dynkin diagrams. However, according to A.~J.~Coleman \cite[p.~450]{Co89}, they first appeared in mimeo\-graphed notes written by H.~S.~M.~Coxeter (around 1935). Therefore we call them Coxeter-Dynkin diagrams.
 
 \begin{example} \label{ex:A_kCD}
 We continue Example~\ref{ex:A_k}. The Coxeter-Dynkin diagram with respect to $(\delta_1, \ldots , \delta_k)$ is a complete graph with only dashed edges (i.e., each two vertices are joined by a dashed edge).
 \end{example}

 If $(\delta_1,\ldots,\delta_\mu)$ is a distinguished basis then the classical monodromy operator of
$f$ can be expressed as follows:
$$h_\ast = h_{\delta_1}\cdots h_{\delta_\mu}.$$ 
We call this product the {\em Coxeter element}\index{Coxeter element}\index{element!Coxeter}  corresponding to the distinguished basis. This follows from the fact that the loop $\omega$ corresponding to $h_\ast$ is
homotopic to the combination $\omega_\mu \omega_{\mu-1}\cdots \omega_1$ 
of the simple loops associated to $h_{\delta_\mu},h_{\delta_{\mu-1}},\ldots,h_{\delta_1}$.

We have the following algebraic proposition (cf.\ \cite[Ch.~V, \S 6, Exercice~3]{Bou02}).

\begin{proposition} \label{prop:Bou}
Let $M$ be a free $\ZZ$-module of rank $\ell$ with a basis $(e_1, \ldots, e_\ell)$ and $A=(a_{ij})$ an $\ell \times \ell$-matrix with integral coefficients. Consider the operator $s_i : M \to M$ defined by
\[
s_i(e_j)=e_j -a_{ij}e_i
\]
and let $c=s_1 \cdots s_\ell$.
Let $C$ be the matrix of $c$ with respect to the basis $(e_1, \ldots, e_\ell)$, $I$ the $\ell \times \ell$ unit matrix, and let $U=(u_{ij})$ and $V=(v_{ij})$ be the matrices defined by
\[
u_{ij} = \left\{ \begin{array}{ll} a_{ij} & \mbox{if } i<j, \\ 0 & \mbox{otherwise}, \end{array} \right. \quad
v_{ij} = \left\{ \begin{array}{ll} 0 & \mbox{if } i<j, \\ a_{ij} & \mbox{otherwise}. \end{array} \right.
\]
Then
\[
 C =(I+U)^{-1}(I-V).
 \]
\end{proposition}

Let $S_\eps^{2n+1}$ be the boundary of the ball $B_\eps$. The set $K = f^{-1}(0) \cap S_\eps^{2n+1}$ is called the {\em link}\index{link} of the singularity $f$. Let $T$ be an (open) tubular neighborhood of $K$ in $S_\eps^{2n+1}$. Milnor \cite{Mi68} has shown that the map
$$
  \funktion{\Phi}{S_\eps^{2n+1}\setminus T}{S^1\subset\CC}{z}
  {\frac{f(z)}{|f(z)|}}
  $$
is the projection of a differentiable fiber bundle. Moreover, this fibration is equivalent to the restriction of the fibration $f|_{X^\ast} : X^\ast \to \Delta^\ast$ to the boundary $S^1_{\eta}$ of $\overline{\Delta}$ \cite[\S 5]{Mi68}. In particular, the fiber $Z_{w/|w|} := \Phi^{-1}(w/|w|)$ is diffeomorphic to $X_w$ for $w \in S^1_\eta$. Let $g_t: Z_1 \to Z_{e^{2 \pi i t}}$ be the parallel transport along $\omega(t) = e^{2 \pi i t}$. For the definition of the linking number see \cite[2.3]{AGV88}, \cite[4.7]{Eb07}.
\begin{definition}
The {\em Seifert form}\index{Seifert form}\index{form!Seifert} of $f$ is the bilinear form $L$ on $\widetilde{H}_n(Z_1) \cong \widetilde{H}_n(X_\eta)$ defined by $L(a,b)=l(a, g_{1/2 *}(b))$ where $l( \ , \ )$ is the linking number.
\end{definition}

Let $(\delta_1,\ldots,\delta_\mu)$ be a distinguished basis of $f$ and let
\begin{itemize}
\item $S:=(\langle\delta_i,\delta_j\rangle)_{j=1,\ldots,\mu}^{i=1,\ldots,\mu}$ be the intersection matrix,
\item $L:=(L(\delta_i,\delta_j))_{j=1,\ldots,\mu}^{i=1,\ldots,\mu}$ be the matrix of the Seifert form, and
\item $H$ be the matrix of the monodromy $h_\ast$ with respect to the basis $(\delta_1,\ldots,\delta_\mu)$.
\end{itemize}
Then one has the following theorem.
\begin{theorem} \label{thm:SLH}
The following holds:
\begin{itemize}
\item[{\rm (i)}] The matrix $L$ is a lower triangular matrix with $-(-1)^{n(n-1)/2}$ on the diagonal.
\item[{\rm (ii)}] 
$S = - L - (-1)^n L^t$.
\item[{\rm (iii)}] 
$H= (-1)^{n+1} L^{-1} L^t$.
\end{itemize}
\end{theorem}

\begin{proof}
(i) This is \cite[Lemma~2.5]{AGV88}. (Note that, according to \cite[Remark in 2.5]{AGV88}, the matrix of the bilinear form in \cite{AGV88} is written down as the matrix of the corresponding operator and hence corresponds to the transpose matrix in our convention. See also \cite[Corollary~5.3~(i)]{Eb07}, where, unfortunately, there is a misprint: ''upper'' should be ''lower''.)

For the proof of (ii) see \cite[Theorem~2.4]{AGV88} (see also \cite[Corollary~5.3~(ii)]{Eb07}).

(iii) follows from (i) and (ii) by applying Proposition~\ref{prop:Bou}.
\end{proof}

It follows from Theorem~\ref{thm:SLH} that each of these matrices determines the other two.  It is clear that $S$ and $L$ determine the matrix $H$. That the matrix $H$ of the classical monodromy operator with respect to a distinguished basis determines the intersection matrix $S$ was first proved by F.~Lazzeri \cite{Laz74} and follows from Theorem~\ref{thm:SLH} and a simple fact in linear algebra \cite[Lemma~2.6]{AGV88} (see also \cite[Lemma~5.5]{Eb07}).

A.~B.~Givental \cite{Gi88} introduced $q$-analogues of the invariants and formulas above thus interpolating between the symmetric and skew symmetric versions of these invariants. The  $q$-analogue of the monodromy group was studied by G.~G.~Il'yuta \cite{Il96}.

%%%%%%%%%%%%%%%%%%%%%%
\section{Change of basis} \label{sect:braid}
A distinguished or weakly distinguished system $(\gamma_1,\ldots,\gamma_\mu)$ of paths can be chosen in many various ways.
Next we consider elementary operations on path systems which preserve the property of being distinguished or weakly distinguished.

 Let $(\gamma_1,\ldots,\gamma_\mu)$ be a  distinguished system of paths from  the critical values $s_1,\ldots,s_\mu$ to the non-critical value $\eta$
 and  let $(\delta_1,\ldots,\delta_\mu)$ be a corresponding  distinguished system of vanishing cycles.
 Furthermore, let $(\omega_1,\ldots,\omega_\mu)$ be a corresponding system of simple loops.

 \begin{definition}
 The operation $\alpha_j$\index{$\alpha_j$} for $1\leq j<\mu$ is defined as
  $$
 \alpha_j :(\gamma_1,\ldots,\gamma_\mu) \mapsto (\gamma_1, \ldots , \gamma_{j-1}, \widetilde{\gamma}_j, \widetilde{\gamma}_{j+1}, \gamma_{j+2}, \ldots, \gamma_\mu),  
  $$
  where $\widetilde{\gamma}_{j+1}= \gamma_j$ and ${\widetilde\gamma}_j$ is a small homotopic deformation of $\gamma_{j+1}\omega_j$
 such that ${\widetilde\gamma}_j$ has no self-intersection points and intersects the other paths only at~$\eta$, for $t=1$
 (see Fig.~\ref{abbIII41}).
  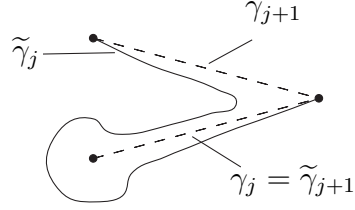
\begin{figure}[h]
  $$
  \unitlength1cm
  \begin{picture}(4,3)
  \put(0.5,0){\includegraphics{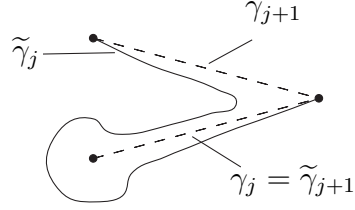}} \put(0.1,1.9){${\widetilde\gamma}_j$}
  \put(3,0.2){$\gamma_j={\widetilde\gamma}_{j+1}$} \put(3.2,2.5){$\gamma_{j+1}$}
  \end{picture}
  $$
  \caption{The operation $\alpha_j$}\label{abbIII41}
  \end{figure}

 Then $({\widetilde\gamma}_1,\ldots,{\widetilde\gamma}_\mu)$ is again a  distinguished system of paths.

This induces the following operation on the corresponding system $(\delta_1,\ldots, \delta_\mu)$ of vanishing cycles which will be denoted by the same symbol:
\[
\alpha_j: (\delta_1,\ldots, \delta_\mu) \mapsto (\delta_1, \ldots , \delta_{j-1}, h_{\delta_j}(\delta_{j+1}), \delta_j, \delta_{j+2} , \ldots, \delta_\mu)
\]
where 
$$
h_{\delta_j}(\delta_{j+1})=\delta_{j+1}-(-1)^{n(n-1)/2} \langle\delta_{j+1},\delta_j\rangle\delta_j.
$$
\end{definition}

 \begin{definition}
 The operation $\beta_{j+1}$\index{$\beta_{j+1}$} for $1\leq j<\mu$ is defined as
  $$
  \beta_{j+1}: (\gamma_1,\ldots,\gamma_\mu) \mapsto (\gamma_1, \ldots, \gamma_{j-1}, \gamma_j', \gamma_{j+1}', \gamma_{j+2}, \ldots , \gamma_\mu)
  $$
 where $\gamma_j'=\gamma_{j+1}$ and
 $\gamma_{j+1}'$ is a small homotopic deformation of $\gamma_j \omega_{j+1}^{-1}$ with the properties above (see Fig.~\ref{abbIII42}).
  \begin{figure}[h]
  $$
  \unitlength1cm
  \begin{picture}(4,3)
  \put(0.5,0){\includegraphics{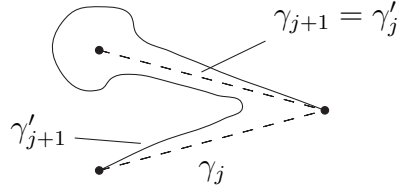}} \put(0,0.5){$\gamma_{j+1}'$}
  \put(2.5,0){$\gamma_j$} \put(3.5,2){$\gamma_{j+1}=\gamma_j'$}
  \end{picture}
  $$
  \caption{The operation $\beta_{j+1}$}\label{abbIII42}
  \end{figure}
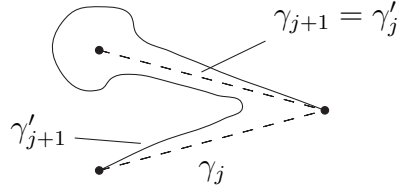
 Then $(\gamma_1',\ldots,\gamma_\mu')$ is again a  distinguished system of paths.
 
 This induces the following operation on the corresponding system $(\delta_1,\ldots, \delta_\mu)$ of vanishing cycles which will also be denoted by the same symbol:
 \[ 
 \beta_{j+1} : (\delta_1,\ldots, \delta_\mu) \mapsto (\delta_1, \ldots , \delta_{j-1}, \delta_{j+1}, h_{\delta_{j+1}}^{-1}(\delta_j), \delta_{j+2} , \ldots, \delta_\mu)
 \]
 where 
 $$
 h_{\delta_{j+1}}^{-1}(\delta_j)=\delta_j-(-1)^{n(n-1)/2}  \langle\delta_{j+1},\delta_j\rangle\delta_{j+1}
 $$
 is the inverse Picard-Lefschetz transformation.
 \end{definition}

 Two  distinguished systems $(\gamma_1,\ldots,\gamma_\mu)$ and $(\tau_1,\ldots,\tau_\mu)$ of paths
 are called {\em homotopic}\index{homotopic}
 if there are homotopies $\phi_i:I\times I\to\bar\Delta$ between $\gamma_i$ and $\tau_i$,
 $i=1,\ldots,\mu$,
 such that for all $u\in I$ and paths $\phi_i^u:I\to \bar\Delta$, $t\mapsto \phi_i(u,t)$, $i=1,\ldots,\mu$,
 the following properties are satisfied:

 (i) $\phi_i^u(0)=s_i$, $\phi_i^u(1)=\eta$.

 (ii) The paths $\phi_i^u$ are double point free.

 (iii) Each two paths $\phi_i^u$ and $\phi_j^u$ have, for $i\neq j$, only the end point $\eta$ in common.

 One can easily show (see \cite[Lemma~6]{Eb07}):

 \begin{lemma}\label{IIIlem91}
 The operations $\alpha_j$ and $\beta_{j+1}$ are mutually inverse,
 i.e.,  the application of $\alpha_j\beta_{j+1}$ and $\beta_{j+1}\alpha_j$ to a 
 distinguished path system $(\gamma_1,\ldots,\gamma_\mu)$ yields a homotopic  distinguished path system.
 \end{lemma}
 
 %\begin{satz}\label{IIIsatz91}
 Up to homotopy of  distinguished path systems we have

 $\on{(i)}$ $\alpha_i\alpha_j=\alpha_j\alpha_i$ for $i,j$ with
 $|i-j|\geq2$,

 $\on{(ii)}$ $\alpha_j\alpha_{j+1}\alpha_j=\alpha_{j+1}\alpha_j \alpha_{j+1}$
 for $1\leq j<\mu-1$.
% \end{satz}

These are the relations of Artin's braid group \cite{Art25, Art47} (see also \cite{Bir74}). Therefore we have an action of the braid group ${\rm Br}_\mu$ on $\mu$ strings
 on the set of the homotopy classes of 
 distinguished path systems and so also on the set of all 
 distinguished systems of vanishing cycles. One can show the following result (\cite{GZ77}, see also \cite[Proposition~5.15]{Eb07}).

\begin{proposition} \label{prop:transitive}
The braid group ${\rm Br}_\mu$ acts transitively on the set of all homotopy classes of distinguished path systems, i.e., any two distinguished path systems can be transformed one to the other by iteration of the operations $\alpha_j$ and $\beta_{j+1}$ and a succeeding homotopy.
 \end{proposition}
 
\begin{definition}
Let 
\begin{itemize}
\item $\calB$\index{$\calB$} be the set of all distinguished bases of vanishing cycles of $f$, 
\item $\calD$\index{$\calD$} be the set of Coxeter-Dynkin diagrams of distinguished bases of $f$.
\end{itemize}
\end{definition}

One also has a braid group action on the sets $\calB$ and $\calD$. 
Moreover, one can change the orientation of a cycle. Let $H_\mu$ be the direct product of $\mu$ cyclic groups of order two with generators $\kappa_1, \ldots , \kappa_\mu$, where $\kappa_i$\index{$\kappa_i$} acts on $\calB$ by
\[
\kappa_i: (\delta_1, \ldots, \delta_i , \ldots, \delta_\mu) \mapsto (\delta_1, \ldots, -\delta_i, \ldots , \delta_\mu).
\]
The braid group ${\rm Br}_\mu$ acts on $H_\mu$ by permutation of the generators $\kappa_1, \ldots , \kappa_\mu$: $\alpha_j$ corresponds to the transposition of $\kappa_j$ and $\kappa_{j+1}$. Let ${\rm Br}_\mu^\rtimes=H_\mu \rtimes {\rm Br}_\mu$\index{${\rm Br}_\mu^\rtimes$} be the semi-direct product. 
It follows from Proposition~\ref{prop:transitive} that the action of the group ${\rm Br}_\mu^\rtimes$ on $\calB$ is transitive.

The set $\calB$ depends on the chosen morsification. In order to get an invariant of the singularity, Brieskorn \cite{Br83} proposed a more general notion of distinguished bases. Namely, he considered the natural action of the monodromy group $\Gamma$ on the set $\calB$: An element $h \in \Gamma$ acts as follows:
\[
h: (\delta_1, \ldots , \delta_\mu) \mapsto (h(\delta_1), \ldots , h(\delta_\mu)).
\]
Brieskorn called a basis $B$ of $M$ {\em geometric}\index{geometric basis}\index{basis!geometric} if it is obtained by any choice of a distinguished path system, of orientations, and of $h \in \Gamma$. He introduced the notions 
\begin{itemize}
\item  $\calB^*$\index{$\calB^*$} for the set of all geometric bases of $f$,
\item $\calD^*$\index{$\calD^*$} for the set of Coxeter-Dynkin diagrams of geometric bases of $f$.
\end{itemize}
The sets $\calB^*$ and $\calD^*$ are invariants of the singularity. In fact, the set $\calD^*$ coincides with $\calD$. The action of $\Gamma$ commutes with the action of the group ${\rm Br}_\mu^\rtimes$.
It follows from Proposition~\ref{prop:transitive} that the action of the group $\Gamma \times {\rm Br}_\mu^\rtimes$ on $\calB^*$ is transitive. One can derive from this that the invariants $\calB^*$ and $\calD^*$ determine each other, see \cite{Br83}.

Note that, unfortunately, in \cite{Eb18} the set $\calB$ was considered but denoted by $\calB^*$.
 
The braid group action above first appeared in a paper of A.~Hurwitz \cite{Hur91} from 1891 where he describes a braid group action on certain sets of Riemann surfaces (cf.\ \cite{Kl88}). It was also studied by Brieskorn and his students, see \cite{Br88}. In \cite{Br88}, Brieskorn introduced a simple unifying concept, the notion of an automorphic set.

\begin{definition} An {\em automorphic set}\index{automorphic set}\index{set!automorphic} is a set $\Lambda$ with a product $\ast : \Lambda \times \Lambda \to \Lambda$ such that all left translations are automorphisms, i.e., one has the following properties:
\begin{itemize}
\item[(i)] For all $a,c \in \Lambda$ there is a unique $b \in \Lambda$ such that $a \ast b=c$.
\item[(ii)] For all $a,b,c \in \Lambda$ one has $(a \ast b) \ast (a \ast c) =a \ast (b \ast c)$.
\end{itemize}
\end{definition}

The set $\Lambda^*$ of vanishing cycles of $f$ is an automorphic set with the product $a \ast b := h_a(b)$ for $a,b \in \Lambda^*$.

If $\Lambda$ is an automorphic set, then one has a canonical braid group action on the $n$-fold cartesian product $\Lambda^n$ of $\Lambda$:
\[
\alpha_i: (x_1, \ldots, x_n) \mapsto (x_1, \ldots, x_{i-1}, x_i \ast x_{i+1}, x_i, x_{i+2}, \ldots, x_n).
\]
The concept of an automorphic set is a basic concept which is also studied under the following names: left self-distributive system, self-distributive groupoid, quandle, wrack, and rack, see, e.g., the book of P.~Dehornoy \cite{Deh00}.

This braid group action is also considered in the representation theory of algebras, see, e.g.,  \cite{CB93, KY11, Rin94}. It has also been applied in mathematical physics, see, e.g., \cite{CV93, FHHI03, FH03}.

\begin{example} \label{ex:A_kclass}
We continue Example~\ref{ex:A_kCD}. By the transformations
\[
\alpha_{k-1}, \alpha_{k-2}, \ldots , \alpha_1; \alpha_{k-1}, \alpha_{k-2}, \ldots, \alpha_2; \ldots ; \alpha_{k-1}, \alpha_{k-2}; \alpha_{k-1},
\]
the distinguished basis $(\delta_1, \ldots , \delta_k)$ is transformed to a distinguished basis  with the Coxeter-Dynkin diagram depicted in Fig.~\ref{fig:CD2}. This is the classical Coxeter-Dynkin diagram of type $A_k$.
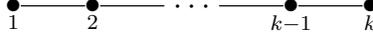
\begin{figure}[h]
$$
\xymatrix{ 
 *{\bullet} \ar@{-}[r]  \ar@{}_{1} & *{\bullet} \ar@{-}[r]  \ar@{}_{2} & \cdots &  *{\bullet} \ar@{-}[l] \ar@{-}[r]  \ar@{}_{k-1} & *{\bullet}  \ar@{}_{k}                   
 }
$$Ê
\caption{Standard Coxeter-Dynkin diagram $A_k$} \label{fig:CD2}
\end{figure}
\end{example}
 
 Finally, we consider operations that transform weakly distinguished path systems again
 into weakly distinguished path systems.

 Let $(\gamma_1,\ldots,\gamma_\mu)$ now be a weakly distinguished path system from the points $s_1,\ldots,s_\mu$ to $\eta$,
 let $(\omega_1,\ldots,\omega_\mu)$ be a corresponding system of simple loops and let
 $(\delta_1,\ldots,\delta_\mu)$ be a corresponding weakly distinguished system of vanishing cycles.

 \begin{definition}
 We define operations $\alpha_i(j)$\index{$\alpha_i(j)$} and $\beta_i(j)$ \index{$\beta_i(j)$}for
 $i,j\in\{1,\ldots,\mu\}$, $i\neq j$, as follows:
  \begin{eqnarray*}
  \alpha_i(j) & : & (\gamma_1,\ldots,\gamma_\mu) \mapsto
  (\gamma_1,\ldots,\gamma_{j-1},\gamma_j\omega_i,\gamma_{j+1},\ldots,\gamma_\mu)\, , \\
  \beta_i(j) & : & (\gamma_1,\ldots,\gamma_\mu) \mapsto
  (\gamma_1,\ldots,\gamma_{j-1},\gamma_j\omega_i^{-1},\gamma_{j+1},
  \ldots,\gamma_\mu).
  \end{eqnarray*}
 \end{definition}

 These operations induce the following operations on the corresponding systems of simple loops,
 and we denote them by the same symbols:
  \begin{eqnarray*}
  \alpha_i(j) & : &  (\omega_1,\ldots,\omega_\mu) \mapsto
  (\omega_1,\ldots,\omega_{j-1},\omega_i^{-1}\omega_j\omega_i,
  \omega_{j+1},\ldots,\omega_\mu) \, ,\\
  \beta_i(j) & : & (\omega_1,\ldots, \omega_\mu) \mapsto
  (\omega_1,\ldots,\omega_{j-1},\omega_i\omega_j\omega_i^{-1},\omega_{j+1},
  \ldots,\omega_\mu).
  \end{eqnarray*}
 If $\omega_1,\ldots,\omega_\mu$ forms a generating system for
 $\pi_1(\Delta',\eta)$,
 then $\pi_1(\Delta',\eta)$ is also generated by the new simple loops that arise from application of the operations $\alpha_i(j)$ and
 $\beta_i(j)$.
 Hence $\alpha_i(j)$ and $\beta_i(j)$ transfer weakly distinguished path systems again to weakly distinguished path systems.

 These operations thus induce operations on the corresponding weakly distinguished systems of vanishing cycles too,
 which we denote by the same symbols, and they appear as follows:
  \begin{eqnarray*}
  \alpha_i(j) & : & (\delta_1,\ldots,\delta_\mu) \mapsto
  (\delta_1,\ldots,\delta_{j-1},h_{\delta_i}(\delta_j),\delta_{j+1},\ldots,\delta_\mu),\\
  \beta_i(j) & : & (\delta_1,\ldots,\delta_\mu) \mapsto
  (\delta_1,\ldots,\delta_{j-1},h_{\delta_i}^{-1}(\delta_j),
  \delta_{j+1},\ldots,\delta_\mu).
  \end{eqnarray*}

 The operations $\alpha_i(j)$ and $\beta_i(j)$ are again mutually inverse in the sense above.
 For even $n$ they even agree.

 If $(\gamma_1,\ldots,\gamma_\mu)$ is a  distinguished path system and
 if $\tau_{j,j+1}\in \calS_\mu$ denotes the transposition of $j$ and $j+1$, then, up to
 homotopy,
  \begin{eqnarray*}
  \alpha_j&=&\tau_{j,j+1}\circ \alpha_j(j+1),\\
  \beta_{j+1}&=&\tau_{j,j+1}\circ \beta_{j+1}(j).
  \end{eqnarray*}

 We now also have the following proposition:

\begin{proposition}\label{prop:wdtransitive}
 Let $(\omega_1,\ldots,\omega_\mu)$ and $(\omega_1',\ldots,\omega_\mu')$ be two free generating systems of the free group $\pi_1(\Delta',\eta)$ such
 that $\omega_i$ and $\omega_i'$ are conjugate to one another for $i=1,\ldots,\mu$.
 Then one can obtain $(\omega_1',\ldots,\omega_\mu')$ from $(\omega_1,\ldots,\omega_\mu)$
 by the application of a sequence of operations of type $\alpha_i(j)$ or $\beta_i(j)$.
\end{proposition}

 This proposition was conjectured by Gusein-Zade \cite{GZ77} and proved by S.~P.~Humphries \cite{Hum85} in 1985.
 It also follows, as remarked by R.~Pellikaan, from an old result of J.~H.~C.~Whitehead from the year 1936
 (cf.~\cite[Proposition~4.20]{LS77}). We refer to \cite{Hum85}.

It follows from Proposition~\ref{prop:wdtransitive} that
 any two weakly distinguished systems of vanishing cycles can be transformed one to the other by iteration of
 the operations $\alpha_i(j)$ and $\beta_i(j)$ and a succeeding change of
 orientation of some of the cycles.

J.~McCool \cite{McC86} found a presentation of the subgroup of the automorphism group of a free group generated by the operations $\alpha_i(j)$ and $\beta_i(j)$.

%%%%%%%%%%%%%%%%%%%%%%
\section{Computation of intersection matrices} \label{sect:comp}
The {\em Sebastiani-Thom sum}\index{Sebastiani-Thom sum}\index{sum!Sebastiani-Thom} of the singularities $f: (\CC^{n+1},0) \to (\CC,0)$ and $g: (\CC^m,0) \to (\CC,0)$ is the singularity of the function germ $f \oplus g: (\CC^{n+m+1},0) \to (\CC,0)$ defined by the formula
\[
(f \oplus g)(x,y)=f(x)+g(y)
\]
($x \in \CC^{n+1}, y \in \CC^m, (x,y) \in \CC^{n+m+1} \cong \CC^{n+1} \oplus \CC^m$).

M.~Sebastiani and R.~Thom \cite{ST71} proved that the monodromy operator of the singularity $f \oplus g$ is equal to the tensor product of the monodromy operators of the singularities $f$ and $g$. If $L_f$, $L_g$, and $L_{f \oplus g}$ denote the Seifert form of $f$, $g$, and $f \oplus g$ respectively, then by a result of Deligne (see \cite{Dem75}, see also \cite[Theorem~2.10]{AGV88})
\[
L_{f \oplus g} = (-1)^{(n+1)m} L_f \otimes L_g.
\]
Gabrielov \cite{Ga73} showed how to calculate an intersection matrix of $f \oplus g$ from the intersection matrices of $f$ and $g$ with respect to distinguished bases (see also \cite[Theorem~2.11]{AGV88}). As a corollary, he obtained certain intersection matrices for singularities of the form
\[ f(x)=z_0^{a_0}+ \cdots +z_n^{a_{n}}, \quad \mbox{for } a_i \in \ZZ,\  a_i \geq 2,\  i=0, \ldots , n.
\]
These singularities are called {\em Brieskorn-Pham singularities}\index{Brieskorn-Pham singularities}\index{singularities!Brieskorn-Pham}. They were considered by Brieskorn \cite{Br66} and Pham \cite{Ph65} (see also \cite{HM68}).
For such a singularity, already Pham \cite{Ph65} had found a basis and calculated the intersection matrix with respect to this basis. Gabrielov showed that Pham's basis can be deformed to a distinguished basis and the intersection matrix is given by the same formulas which Gabrielov obtained. Independently, these intersection matrices with respect to distinguished bases were also  calculated by A.~Hefez and Lazzeri \cite{HL74}.

A special case of the Sebastiani-Thom sum of $f$ and $g$ is the case when $g(y)=y_1^2+ \cdots + y_m^2$. This is called a {\em stabilization}\index{stabilization}  of $f$. The following theorem is a special case of Gabrielov's result (see also \cite[Theorem~2.14]{AGV88}).

\begin{sloppypar}

 \begin{theorem} \label{theo:stab}
 Let $f_\lambda$ be a morsification of the singularity $f$,
 let $(\gamma_1,\ldots,\gamma_\mu)$ be a  distinguished path system for
 $f_\lambda$,
 and let $(\delta_1,\ldots,\delta_\mu)$ be a corresponding  distinguished basis.

 Then $f_\lambda(x)+y_1^2+\ldots+y_m^2$ is a morsification of the singularity
 $f(x)+y_1^2+\ldots+y_m^2$, with the same critical values,
 $(\gamma_1,\ldots,\gamma_\mu)$ is also a  distinguished path system for this singularity,
 and for a corresponding  distinguished basis
 $({\widetilde\delta}_1,\ldots,{\widetilde\delta}_\mu)$ we have
  $$
  \langle {\widetilde\delta}_i,{\widetilde\delta}_j\rangle=[\on{sign}(j-i)]^m
  (-1)^{(n+1)m+\frac{m(m-1)}{2}}\langle\delta_i,\delta_j\rangle \ \mbox{for } i\neq j.
  $$
 \end{theorem}

\end{sloppypar}

It follows from Theorem~\ref{theo:stab} that, by taking a suitable stabilization, one can assume that $n \equiv 2 \, \mbox{mod}\, 4$. In this case, the intersection form is symmetric and the vanishing cycles have self intersection number $-2$. The Picard-Lefschetz transformation $h_{\delta_i}$ acts on $M$ by the formula
\[
h_{\delta_i}(\alpha)=s_{\delta_i}(\alpha)=\alpha + \langle \alpha , \delta_i \rangle \delta_i.
\]
This is a reflection in the hyperplane orthogonal to the vanishing cycle $\delta_i$. In accordance with the definition in Sect.~\ref{sect:CD}, in the Coxeter-Dynkin diagram, edges of weight $+1$ are depicted by solid lines and edges of weight $-1$ are depicted by dashed lines. Note that the definition of a Coxeter-Dynkin diagram in \cite[2.8]{AGV88} is slightly different: It encodes the intersection matrix in the case $n \equiv 2 \, \mbox{mod}\, 4$,  and the $i$-th and $j$-th vertices are joined by an edge of multiplicity $\langle \delta_i, \delta_j \rangle$. 

\begin{example} Consider the germ of the function $f: \CC^2 \to \CC$ defined by  $f(x,y)=x^5+y^3$. (This is the singularity $E_8$, see Sect.~\ref{sect:special}.) By Example~\ref{ex:A_kclass} and the result of Gabrielov \cite{Ga73} there is a distinguished basis of $f$ with a Coxeter-Dynkin diagram of the shape of Fig.~\ref{fig:E_8Gab}.
\begin{figure}[h]
$$
\xymatrix{ 
 & *{\bullet} \ar@{-}[r] \ar@{}^{2} \ar@{-}[d]    & *{\bullet} \ar@{-}[r] \ar@{}^{4} \ar@{-}[d]  &  *{\bullet}  \ar@{-}[r] \ar@{}^{6} \ar@{-}[d]  &  *{\bullet} \ar@{}[r] 
\ar@{}^{8} \ar@{-}[d]  &  \\               
&  *{\bullet} \ar@{-}[r] \ar@{}_{1} \ar@{--}[ur]   & *{\bullet} \ar@{-}[r] \ar@{}_{3} \ar@{--}[ur]   &  *{\bullet}  \ar@{-}[r]  \ar@{}_{5} \ar@{--}[ur]   & *{\bullet}  \ar@{}[r] 
 \ar@{}_{7}   &               
 }
$$Ê
\caption{Gabrielov diagram of $E_8$} \label{fig:E_8Gab}
\end{figure}
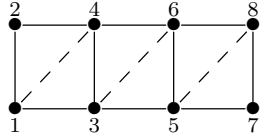
By the transformations
\[
\alpha_7, \alpha_6, \alpha_5, \alpha_4, \alpha_3, \alpha_2, \alpha_1; \beta_5, \beta_4;  \beta_7, \beta_6, \beta_5;   \beta_7, \beta_6, \beta_5; \beta_8, \beta_7, \beta_6; \kappa_2, \kappa_7, \kappa_8,
\]
the Coxeter-Dynkin diagram is transformed to the classical Coxeter-Dynkin diagram of type $E_8$, see Fig.~\ref{fig:E_8class}. It follows from Theorem~\ref{thm:Deligne} below that the numbering can be changed by braid group transformations to an arbitrary numbering.
\begin{figure}[h]
$$
\xymatrix{ 
*{\bullet} \ar@{-}[r] \ar@{}^{7}  & *{\bullet} \ar@{-}[r] \ar@{}^{6} &  *{\bullet}  \ar@{-}[r] \ar@{}^{5} \ar@{-}[d]  &  *{\bullet} \ar@{-}[r]   \ar@{}^{4} & *{\bullet} \ar@{-}[r]   \ar@{}^{3} & *{\bullet} \ar@{-}[r]   \ar@{}^{2} & *{\bullet}  \ar@{}^{1}\\               
&  & *{\bullet}\ar@{}_{8}  & & & &                 
 }
$$Ê
\caption{Standard Coxeter-Dynkin diagram $E_8$} \label{fig:E_8class}
\end{figure}
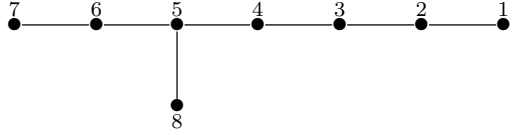
\end{example}

Another method to compute an intersection matrix with respect to a distinguished basis of $f$ is the polar curve method of Gabrielov \cite{Ga79}. 

If $n=1$, so $f:(\CC^2,0) \to (\CC,0)$ defines a curve singularity, there is an especially nice method to compute an intersection matrix with respect to a distinguished basis using a real morsification of the singularity. This method is independently due to N.~A'Campo \cite{A'C75a} and Gusein-Zade \cite{GZ74, GZ75}.

P.~Orlik and R.~Randell \cite{OR77} computed the classical monodromy operator for  weighted homogeneous polynomials of the form
\[
f(z_0, \ldots , z_n)=z_0^{a_0} +z_0z_1^{a_1} + \ldots + z_{n-1}z_n^{a_n}, \quad n \geq 1.
\]
Moreover, they formulated the following conjecture. Let $r_k=a_0a_1 \cdots a_k$ for $k=0,1, \ldots ,n$, $r_{-1}=1$, and define integers $c_0, c_1, \ldots , c_\mu$ by
\[
\prod_{i=-1}^n(t^{r_i}-1)^{(-1)^{n-i}} = c_\mu t^{\mu} + \cdots + c_1t+c_0.
\]

\begin{conjecture}[Orlik-Randell]
There exists a distinguished basis of $f$ such that the Seifert matrix $L$ of $f$ is given by
\[
L =  -(-1)^{n(n+1)/2} \begin{pmatrix} c_0 & 0 & \cdots & \cdots & 0 & 0 & 0 \\
                                                              c_1 & c_0 & 0 &  \cdots & \cdots & 0 & 0\\
                                                              c_2 & c_1 & c_0 & 0 & \cdots & \cdots & 0 \\
                                                              \vdots & \vdots & \vdots & \ddots & \ddots & \vdots & \vdots\\
                                                              c_{\mu-3}  & c_{\mu-4} & c_{\mu-5} & \cdots & c_0 & 0 & 0 \\
                                                              c_{\mu-2} & c_{\mu-3}  & c_{\mu-4}& \cdots  & c_1 & c_0 & 0 \\
                                                              c_{\mu-1} & c_{\mu-2} & c_{\mu-3} & \cdots & c_2 & c_1 & c_0
                                       \end{pmatrix}  \, .                       
\]
\end{conjecture}

This conjecture is still open. However, recently D.~Aramaki and A.~Takahashi \cite{AT19} proved an algebraic analogue of this conjecture.

%%%%%%%%%%%%%%%%%%%%%%%%%%%%%
\section{The discriminant and the Lyashko-Looijenga map} \label{sect:LL}
Let $f:(\CC^{n+1},0)\to (\CC,0)$ be a holomorphic function germ with an isolated singularity at $0$, $\on{grad}f(0)=0$.
Then one obtains a universal unfolding $F$ of $f$ as follows (see \cite[1.3]{Gr19} or \cite[Proposition~3.17]{Eb07}):
 Let $g_0=-1,g_1,\ldots,g_{\mu-1}$ be representatives of a basis of the $\CC$-vector space
  $$
  \calO_{n+1}\left/\left(\frac{\partial f}{\partial z_0},\ldots, \frac{\partial
  f}{\partial z_n}\right)\calO_{n+1},\right.
  $$
 which has dimension $\mu$. Then put
  $$
  \funktion{F}{(\CC^{n+1}\times\CC^\mu,0)}{(\CC,0)}{(z,u)}
  {f(z)+\sum\limits_{j=0}^{\mu-1}g_j(z)u_j.}
  $$
 Let
  $$F: V \times U \to \CC
  $$
 be a representative of the unfolding $F$,
 where $V$ is an open neighborhood of $0$ in $\CC^{n+1}$ and $U$ is an open neighborhood of $0$ in $\CC^\mu$.
 We put
  \begin{eqnarray*}
  \calY&:=&\{(z,u)\in V\times U\,|\, F(z,u)=0\},\\
  \calY_u&:=&\{z\in V\ |\ F(z,u)=0\}.
  \end{eqnarray*}
 Since $F(z,0)=f(z)$, there is an $\eps>0$ such that every sphere $S_\rho\subset V$ around $0$
 of radius $\rho\leq\eps$ intersects the set $\calY_0$ transversally.
 Let $\eps>0$ be so chosen.
 Then there is also an $\theta>0$ such that for $|u|\leq\theta$ the set $\{u \in \CC^\mu \, |\,|u|\leq\theta \}$
 lies entirely in $U$ and $\calY_u$ intersects the sphere $S_\eps$ transversally.
 Let $\theta$ be so chosen. We put
  \begin{eqnarray*}
  \calX^\circ&:=&\{(z,u)\in\calY\ \big|\ |z| < \eps,|u|<\theta\},\\
  \calX&:=&\{(z,u)\in\calY\ \big|\ |z|\leq\eps,|u|<\theta\},\\
  \partial\calX&:=&\{(z,u)\in\calY\ \big|\ |z|=\eps,|u|<\theta\},\\
  S&:=&\{u\in U\ \big|\ |u|<\theta\},
  \end{eqnarray*}
  $$
  \funktion{p}{\calX}{S}{(z,u)}{u.}
  $$

 Let $C$ be the set the critical points and $D=p(C)\subset S$ the {\em discriminant}\index{discriminant} of $p$.
 We have the following result (see, e.g., \cite[Proposition~3.21]{Eb07}).
 \begin{theorem} \label{thm:irr}
 For a suitable $\theta>0$ we have:
\begin{itemize}
\item[{\rm (i)}]  The map $p:\calX\to S$ is proper.
\item[{\rm (ii)}]  $C$ is a nonsingular analytic subset of $\calX^\circ$ and is closed in
 $\calX$.
 \item[{\rm (iii)}] The restriction $p|_C:C\to S$ is finite (i.e., proper with finite
 fibers).
\item[{\rm (iv)}] The discriminant $D$ is an irreducible hypersurface in $S$.
\end{itemize}
\end{theorem}

 By the Ehresmann fibration theorem the map
  $$
  p':=p|_{\calX \setminus p^{-1}(D)}:\calX \setminus p^{-1}(D)\longrightarrow S \setminus D
  $$
 is then the projection of a differentiable fiber bundle, and the fibers of $p'$ are
 diffeomorphic to a Milnor fiber ${X}_\eta$ of $f$.

 Let $s\in S \setminus D$ and ${X}_s:=p^{-1}(s)=(p')^{-1}(s)$.
 Then $p'$ defines a representation
  $$
  \rho:\pi_1(S \setminus D,s)\longrightarrow\on{Aut}(\widetilde{H}_n({X}_s)).
  $$
 One has the following theorem (see \cite[Theorem~3.1]{AGV88} or \cite[Proposition~5.17]{Eb07}).

 \begin{theorem} \label{thm:bif}
 The image $\Gamma$ of the homomorphism
  $$
  \rho:\pi_1(S \setminus D,s)\longrightarrow\on{Aut}(\widetilde{H}_n({X}_s))
  $$
 coincides with the monodromy group of the singularity.
 \end{theorem}
 
In particular, Theorem~\ref{thm:bif} implies that the monodromy group is independent of the chosen morsification.

\begin{sloppypar}

As a corollary of the irreducibility of the discriminant (Theorem~\ref{thm:irr}(iv)) and Theorem~\ref{thm:bif} we obtain the following result which was first proved by Gabrielov \cite{Ga74a} and independently by Lazzeri \cite{Laz73, Laz74}. 

\end{sloppypar}

\begin{corollary}[Gabrielov, Lazzeri] \label{cor:connected}
The Coxeter-Dynkin diagram with respect to a weakly distinguished system $(\delta_1,\ldots,\delta_\mu)$ of vanishing cycles is a
connected graph.
\end{corollary}

One can show that if 0 is neither a regular nor a non-degenerate critical point of $f$, then there are two vanishing cycles $\delta, \delta'$ of $f$ with $\langle \delta, \delta' \rangle =1$. This follows from the following result due to G.~N.~Tyurina \cite[Theorem~1]{Ty68} and D.~Siersma \cite[Proposition~(8.9)]{Si74} (see also  \cite[Theorem~3.23]{AGV88}, \cite[5.9]{Eb07}) and the fact that, if 0 is neither a regular nor a non-degenerate critical point of $f$, then $f$ deforms to the singularity $g$ with $g(z)=z_0^3+z_1^2+ \cdots + z_n^2$.

\begin{theorem}[Tyurina, Siersma] 
Let $f_t : (\CC^{n+1},0) \to (\CC,0)$, $t \in [0,1]$, be a continuous deformation of the singularity $f_0=f$ with $\mu(f_0)=\mu$, $\mu(f_t)=\mu'$ for $0 < t \leq 1$. Then one has $\mu \geq \mu'$, one has a natural inclusion of the Milnor lattice $M_{f_t}$ of $f_t$ in the Milnor lattice $M_{f_0}$ of $f_0$, and a distinguished basis of $f_t$ can be extended to a distinguished basis of $f_0$.
\end{theorem}

From these results one obtains another corollary of the irreducibility of the discriminant (cf.\ \cite[Theorem~3.4]{AGV88}, \cite[Proposition~5.20]{Eb07}): 

\begin{corollary} \label{cor:vanishing}
If not both {\rm (i)} $n$ is odd and {\rm (ii)} $0$ is a non-degenerate critical point of
 $f$,
then the set of vanishing cycles $\Lambda^*$ is the only $\Gamma$-orbit, i.e.,  the monodromy group $\Gamma$ acts transitively on $\Lambda^\ast$.
\end{corollary}

Using these results, K.~Saito \cite{Sa82} showed that the monodromy group $\Gamma$ determines the Milnor lattice $M$.

We can also deduce from these results that the classical monodromy operator acts irreducibly (cf.\ \cite[Theorem~3.5]{AGV88}). (An earlier result for curves was obtained by C.~H.~Bey \cite{Bey72a, Bey72b}.)

\begin{corollary}
Let $(\delta_1, \ldots, \delta_\mu)$ be a distinguished basis of $f$ and let $I$ be a subset of the set of indices $I_0=\{1, \ldots , \mu\}$ such that the linear span of the basis elements $\delta_i$ with $i \in I$ is invariant under the classical monodromy operator $h_*$. Then either $I = \emptyset$ or $I=I_0$.
\end{corollary}

\begin{corollary}
If the classical monodromy operator of a singularity is the multiplication by $\pm 1$, then the singularity is non-degenerate.
\end{corollary}

This was first proved by A'Campo \cite[Th\'eor\`eme~2]{A'C73} as an answer to a question of Sebastiani. It was deduced from the following result.

\begin{theorem}[A'Campo]
The trace of the classical monodromy operator of $f$ is 
\[
{\rm tr}\, h_* =(-1)^n.
\]
\end{theorem}

The {\em corank}\index{corank} of a singularity $f$ is the corank of the Hesse matrix of $f$. Using a result of Deligne (see \cite{A'C75c}), the author proved the following result \cite[Proposition~5]{Eb96}.

\begin{proposition}
Let $n \equiv 2\  \mbox{mod}\ 4$ and let $c(f)$ denote the corank of $f$. Then
\[
{\rm tr}\, h^2_* =(-1)^{c(f)}.
\]
\end{proposition} 

A very important result on the classical monodromy is the following theorem.

\begin{theorem}[Monodromy theorem] \label{thm:monodromy}
The classical monodromy of $f$ is quasi-unipotent, i.e., its eigenvalues are roots of unity.
\end{theorem}

For the history of this theorem and further properties of the classical monodromy see the survey article \cite{Eb06}. The usual proofs of Theorem~\ref{thm:monodromy} use a resolution of the singularity, see e.g.\ \cite{EG07} for an instructive one. For a proof which does not use a resolution see \cite{Le78}.

The {\em bifurcation variety}\index{bifurcation variety}\index{variety!bifurcation} ${\rm Bif}$ is the set of all $\lambda \in S$ such that $f_\lambda$ does not have $\mu$ distinct critical values.
Looijenga \cite{Loo74} in 1974 and independently Lyashko (in the same year, but his work was only published later in \cite{Ly79, Ly84}) introduced the following mapping: The {\em Lyashko-Looijenga mapping}\index{Lyashko-Looijenga mapping}\index{mapping!Lyashko-Looijenga} ${\rm LL}$ sends a point $\lambda \in S$ to the unordered collection of critical values of the function $f_\lambda$ or, what amounts to the same thing but is sometimes more convenient, to the polynomial which has these critical values as roots. If $\calP^\mu$ denotes the set of monic polynomials of degree $\mu$, then this is the mapping
\[
\funktion{{\rm LL}}{S}{\calP^\mu}{\lambda}{\prod_{i=1}^\mu (t-s_i)}
\]
where $s_1, \ldots , s_\mu$ are the critical values of the function $f_\lambda$. Let $\Sigma \subset \calP^\mu$ denote the discriminant variety in $\calP^\mu$. Then there exists a neighborhood $U \subset S$ of $0 \in S$ such that ${\rm LL}|_{U \setminus {\rm Bif}} : U \setminus {\rm Bif} \to \calP^\mu \setminus \Sigma$ is locally biholomorphic \cite[Theorem~(1.4)]{Loo74}.

%%%%%%%%%%%%%%%%
\section{Special singularities} \label{sect:special}
We shall now consider what is known about these invariants for special classes of singularities.

Let $f,g : (\CC^{n+1},0) \to (\CC,0)$ be holomorphic function germs with an isolated singularity at 0. The germs $f$ and $g$ are called {\em right equivalent}\index{right equivalent}\index{equivalent!right} if $f$ is taken to $g$ under (the germ of) a biholomorphic mapping of the domain space which leaves the origin invariant.
The {\em modality}\index{modality} (or {\em module number}\index{module number}) of $f$
 is the smallest number $m$
 for which there exists a representative $p:\calX\to S$ of the universal unfolding $F:(\CC^{n+1}\times\CC^\mu,0)\to(\CC,0)$ of $f$
 such that for all $(z,u) \in \calX$ the function germs $F_u: (\CC^{n+1},z) \to (\CC, F(z,u))$ given by $F_u(z')=F(z',u)$
 fall into finitely many families of right equivalence classes
 depending on at most $m$ (complex) parameters. 
 Singularities of modality 0,1 and 2 are called  {\em simple}\index{simple singularities}\index{singularities!simple}, {\em unimodal}\index{unimodal singularities}\index{singularities!unimodal} (or {\em
 unimodular})\index{unimodular singularities}\index{singularities!unimodular},
 and {\em bimodal}\index{bimodal singularities}\index{singularities!bimodal} (or {\em bimodular}\index{bimodular singularities})\index{singularities!bimodular}, respectively.
 
 V.~I.~Arnold classified the singularities up to modality 2 \cite{Arn75}. He listed certain normal forms. A normal form determines a class of singularities. This class  corresponds to a {\em $\mu$=const stratum}\index{$\mu$=const stratum}: Any two singularities of a $\mu$=const stratum are $\mu$-equivalent, see Sect.~\ref{sect:top} below. By Proposition~\ref{prop:Gab} below, the class $\calD$ is the same for all singularities of a $\mu$=const stratum. Gabrielov \cite{Ga74a} proved that the dimension of the $\mu$=const stratum is equal to the modality of the singularity.
 Arnold found that in the lists of classes, all the classes are split into series which are now called the Arnold series. However, as Arnold writes in \cite{Arn75}, ``although the series undoubtedly exist, it is not at all clear what a series is''. Let us look at Arnold's classification.
\begin{table}
\begin{center}
\begin{tabular}{rlrlrl}
\hline
$A_k$: & $x^{k+1},\ k \geq 1$ & $D_k$: & $x^2y+y^{k-1}, \ k \geq 4$ & & \\
$E_6$: & $x^3+y^4$ & $E_7$: & $x^3+xy^3$ & $E_8$: & $x^3+y^5$\\
\hline
\end{tabular}
\end{center}
\caption{Simple singularities} \label{tab:simple}
\end{table}

Let us first assume that $f:(\CC^{n+1},0) \to (\CC,0)$ defines a simple singularity. Up to stabilization, the simple singuarities are given by the germs of the functions of Table~\ref{tab:simple}.
There are many characterizations of simple singularities, see \cite{Du79}. They are the only singularities where, for $n \equiv 2 \ \mbox{mod} \, 4$, the intersection form is negative definite \cite[Characterization~B5]{Du79}. They are also the only singularities where the set $\calD$ contains a tree \cite[Characterization~B7]{Du79}, \cite{A'C75b, A'C76}. Moreover, this is also the only case where the monodromy group $\Gamma$ is finite \cite[Characterization~B8]{Du79}. The author \cite{Eb18} has recently shown that they are the only singularities where the set $\calB$ is finite.

Let $f:(\CC^{n+1},0) \to (\CC,0)$ define a simple singularity and $n \equiv 2 \ \mbox{mod} \, 4$. Then $\Lambda^*$ is a root system of type $A_k, D_k, E_k$. (Note that the usual bilinear form of \cite{Bou02} has to be multiplied by $-1$.) The Milnor lattice $M$ is the corresponding root lattice, the group $\Gamma$ is the corresponding Weyl group, and the classical monodromy operator $h_\ast$ is a Coxeter element of the corresponding root system. Let $c \in \Gamma$ be a Coxeter element. Define
\[ 
\Xi_c :=\{ (s_1, \ldots, s_k) \, | \, s_i \in \Gamma \mbox{ reflection},\ s_1 \cdots s_k=c \}.
\]
Deligne \cite{Del74} in a letter to Looijenga (with the help of J.~Tits and D.~Zagier) showed the following theorem.

\begin{theorem}[Deligne] \label{thm:Deligne}
The braid group ${\rm Br}_k$ acts transitively on $\Xi_c$.
\end{theorem}

From this we obtain the following result. 
\begin {corollary} \label{cor:Deligne}
One has
\[
\calB= \{(\delta_1, \ldots, \delta_k) \in (\Lambda^*)^k \, | \, \langle \delta_1, \ldots, \delta_k \rangle_\ZZ=M, s_{\delta_1} \cdots s_{\delta_k}=h_\ast \}.
\]
\end{corollary}
The sets $\Xi_c$, $\calB$, and $\calD$ are finite sets. The cardinality of these sets was calculated in the letter of Deligne (see also \cite{Vo85b, Kl89}).

\begin{example}  For $E_8$ one has $|\calD_{E_8}|=2^83^45^6=324\,000\,000$.
\end{example}

The first published proof of Theorem~\ref{thm:Deligne} is due to D.~Bessis and can be found in \cite{Bes03}. This theorem has been generalized and it has also applications outside of singularity theory, see \cite{BDSW14}. K.~Igusa and R.~Schiffler generalized this result to arbitrary Coxeter groups of finite rank \cite[Theorem~1.4]{IS10} (see also \cite{BDSW14, BGRW17}). Recently, B.~Baumeister, P.~Wegener, and S.~Yahiatene \cite{BWY19} generalized it to certain extended Weyl groups (see below). The theorem has applications in the theory of Artin groups, see \cite{Bes03, Di06}, and in the representation theory of algebras, see \cite{IS10, Ig11, HK16}.

Let $\RR^k$ be a vector space on which the Weyl group $W=\Gamma$ acts in a canonical way and let $\CC^k= \RR^k \otimes_\RR \CC$ be its complexification. The action of $W$ on $\RR^k$ extends in a natural way to an action of $W$ on $\CC^k$. Let $H$ be be union of the complexifications of the reflection hyperplanes of $W$. Let $S$ be the base space and $D$ the discriminant of the universal unfolding of the simple singularity $f$. Then the pair $(S,D)$ is analytically isomorphic (in a neighborhood of the origin) to the pair $(\CC^k/W, H/W)$ \cite{Arn72}. Brieskorn \cite{Br71} proved that the fundamental group of the space $\CC^k/W \setminus H/W$ is  the generalized Brieskorn braid group $\pi$ \cite{Br73a, BS72} of the Weyl group $W$. He conjectured \cite{Br73a} and Deligne \cite{Del72} proved that this space is in fact a $K(\pi,1)$-space. 
(A $K(\pi,1)$-space is a topological space with fundamental group $\pi$ and trivial higher homotopy groups.) 
From this it follows that the complement $S \setminus D$ is a $K(\pi,1)$-space as well (see also \cite[Theorem~3.9]{AGV88}). Brieskorn asked \cite[Probl\`eme~15]{Br73b} whether this is true in general.

Now we consider the Lyashko-Looijenga map in the case of the simple singularities. Looijenga \cite{Loo74} and Lyashko \cite{Ly79, Ly84} showed that the mapping ${\rm LL}|_{U \setminus {\rm Bif}} : U \setminus {\rm Bif} \to \calP^k \setminus \Sigma$ is a covering of degree
\[
d=\frac{k! N^k}{|\Gamma|}
\]
where $N$ is the Coxeter number of the corresponding root system. I.~S.~Livshits \cite{Li81} determined the Galois group of this covering (see also \cite{Yu99}). Let $p \in \calP^k \setminus \Sigma$. It is well known that
\[
\pi_1(\calP^k \setminus \Sigma,p) \cong {\rm Br}_k.
\]
Therefore the complement of the bifurcation variety of a simple singularity is a $K(\pi,1)$, where $\pi$ is a subgroup of index $d$ in the braid group ${\rm Br}_k$. 

Similar questions were also answered for complex reflection groups, see \cite{Bes15, Rip10, Rip12}.

Looijenga already proved Theorem~\ref{thm:Deligne} in the case $A_k$ \cite[Corollary~(3.8)]{Loo74}. Moreover, in this case, he established a correspondence between generic polynomial coverings of the complex sphere and trees with totally ordered edges. By considering a generalized version of the Lyashko-Looijenga mapping, more general combinatorial results were obtained by Arnold \cite{Arn96}, D.~Zvonkine and S.~K.~Lando \cite{ZL99}, and B.~S.~Bychkov \cite{By15}.

By studying the Lyashko-Looijenga mapping, Jianming Yu \cite{Yu96} determined the number of Seifert matrices with respect to distinguished bases of a simple singularity.

Gusein-Zade \cite{GZ80} gave a characterization of distinguished bases for simple singularities. Let $f: (\CC^{n+1},0) \to (\CC,0)$ define a simple singularity of Milnor number $\mu$. He showed that, if $(\delta_1, \ldots, \delta_\mu)$ is an integral basis of the homology group $M$ in which the matrix of the Seifert form is lower triangular, then $(\delta_1, \ldots, \delta_\mu)$ is a distinguished basis of vanishing cycles. The proof is based on the following result: Let $n$ be even. For any vanishing cycle $\delta$ and any distinguished basis $(\delta_1, \ldots, \delta_\mu)$ for $f$, there exists a distinguished basis $(\delta'_1, \ldots, \delta'_\mu)$ with the first element $\delta'_1= \pm \delta$. H.~Serizawa \cite{Se01} showed that the latter result is false for a non-simple singularity.

The next case are the {\em simple elliptic singularities}\index{simple elliptic singularities}\index{singularities!simple elliptic} (see also \cite{Sa74}). These are the singularities $\widetilde{E}_6$, $\widetilde{E}_7$, and $\widetilde{E}_8$. Up to stabilization, they are given by the one-parameter families of Table~\ref{tab:elliptic}.
\begin{table}
\begin{center}
\begin{tabular}{rll}
\hline
$\widetilde{E}_6$: & $x^3+y ^3+z^3+axyz$, & $a^3+27 \neq 0$ \\
$\widetilde{E}_7$: & $x^2+y^4+z^4+ay^2z^2$, & $a^2 \neq 4$ \\
$\widetilde{E}_8$: & $x^2+y^3+z^6+ay^2z^2$, & $4a^3+27 \neq 0$\\
\hline
\end{tabular}
\end{center}
\caption{Simple elliptic singularities} \label{tab:elliptic}
\end{table}
These singularities can be characterized as follows: For $n \equiv 2 \ \mbox{mod} \, 4$, the intersection form is not negative definite but negative semi-definite \cite[Characterization~C5]{Du79}. Therefore, the simple elliptic singularities are also called the {\em parabolic singularities}\index{parabolic singularities}\index{singularities!parabolic}. The monodromy group is not finite but has polynomial growth \cite[Characterization~C6]{Du79}. (For the notions of polynomial and exponential growth, see e.g.\ \cite{Mi68b}. See also Remark~\ref{rem:hyperbolic} below.) The set $\calB$ is infinite but the set $\calD$ is finite \cite{Eb18}.

Let $f:(\CC^{n+1},0) \to (\CC,0)$ define a simple elliptic singularity of type $\widetilde{E}_k$, $k=6,7,8$, and let $n \equiv 2 \ \mbox{mod} \, 4$. If $M$ is a lattice, denote by
 $M^\#=\on{Hom}(M,\ZZ)$ the dual module and by
  $$
  \funktion{j}{M}{M^\#}{v}{l_v\text{ with }l_v(x)=\langle v,x\rangle,x\in M,}
  $$
the canonical homomorphism.
The Milnor lattice $M$ is the orthogonal direct sum of the root lattice $E_k$ and the radical $\ker j$ which is two-dimensional. The set $\Lambda^*$ of vanishing cycles is an extended affine root system of type $E_k^{(1,1)}$ in the sense of Saito \cite{Sa85}.  The monodromy group $\Gamma$ is the semi-direct product of the group $\ker j \otimes j(M)$ and the Weyl group $W(E_k)$ of $E_k$, where the group $\ker j \otimes j(M)$ acts on $M$ by
$(v \otimes w^\#)(x)=x+w^\#(x) v$ and the action of $W(E_k)$ on $\ker j \otimes j(M)$ is trivial on the first factor and canonical on the second one, see \cite[Proposition~(6.7)]{Loo78}. It follows from this description that the monodromy group has polynomial growth.

P.~Kluitmann extended Corollary~\ref{cor:Deligne} to the simple elliptic singularities \cite{Kl86}. He also calculated the cardinality of $\calD$ for $\widetilde{E}_6$ and $\widetilde{E}_7$. In \cite{Jaw86, Jaw88}, P.~Jaworski considered the Lyashko-Looijenga map for the simple elliptic singularities and showed that the complement of the bifurcation variety of a simple elliptic singularity is again a $K(\pi,1)$ for a certain subgroup of the braid group ${\rm Br}_\mu$ \cite[Corollary~2]{Jaw86}. Recently, C.~Hertling and C.~Roucairol \cite{HR18} used a different approach to study the Lyashko-Looijenga map for the simple and simple elliptic singularities and refined and extended the results of Kluitmann and Jaworski.

For the remaining singularities, the sets $\calB$ and $\calD$ are infinite \cite{Eb18}. Let $f: (\CC^{n+1},0) \to (\CC,0)$ be such a singularity. We assume $n \equiv 2 \ \mbox{mod} \, 4$. The only singularities with a hyperbolic intersection form, i.e., an indefinite form with only one positive eigenvalue, are the singularities of the series $T_{p,q,r}$ with $2 \leq p \leq q \leq r$ and $\frac{1}{p}+\frac{1}{q}+\frac{1}{r} < 1$ \cite{Arn73}. Up to stable equivalence, they are given by the one parameter families
\[
T_{p,q,r}: x^p+y^q+z^r+axyz, \ a \neq 0.
\]
They are also called the {\em hyperbolic singularities}\index{hyperbolic singularities}\index{singularities!hyperbolic}. The simple elliptic and hyperbolic singularities are unimodal singularities. Gabrielov \cite{Ga74b} calculated Coxeter-Dynkin diagrams with respect to distinguished bases for the unimodal singularities. According to \cite{Ga74b}, the simple elliptic and hyperbolic singularities have Coxeter-Dynkin diagrams with respect to distinguished bases in the form of Fig.~\ref{FigTpqr}. 
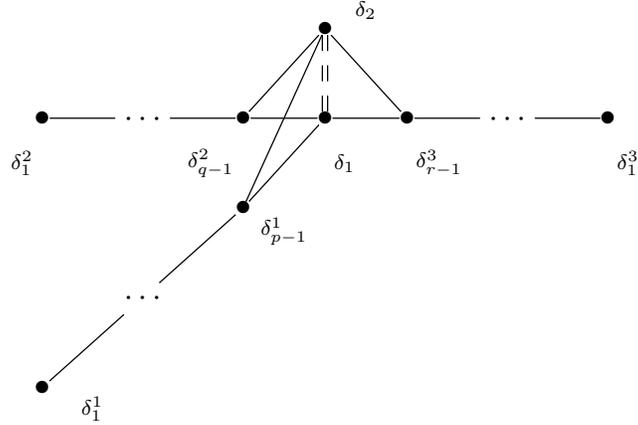
\begin{figure}
$$
\xymatrix{ 
 & & & *{\bullet} \ar@{==}[d] \ar@{-}[dr]  \ar@{-}[ldd] \ar@{}^{\delta_{2}}[r]
 & & &  \\
 *{\bullet} \ar@{-}[r] \ar@{}_{\delta^2_1}[d]  & {\cdots} \ar@{-}[r]  & *{\bullet} \ar@{-}[r] \ar@{-}[ur]   \ar@{}_{\delta^2_{q-1}}[d] & *{\bullet} \ar@{-}[dl] \ar@{-}[r] \ar@{}^{\delta_{1}}[d] & *{\bullet} \ar@{-}[r]  \ar@{}^{\delta^3_{r-1}}[d]  & {\cdots} \ar@{-}[r]  &*{\bullet} \ar@{}^{\delta^3_1}[d]   \\
 & &   *{\bullet} \ar@{-}[dl] \ar@{}_{\delta^1_{p-1}}[r]  & & & & \\
 & {\cdots} \ar@{-}[dl] & & & & & \\
*{\bullet}  \ar@{}_{\delta^1_1}[r] & & & & & &
  } 
$$
\caption{The graph $T_{p,q,r}$} \label{FigTpqr}
\end{figure}
Here $(p,q,r)=(3,3,3), (2,4,4), (2,3,6)$ for $\widetilde{E}_6$, $\widetilde{E}_7$, and $\widetilde{E}_8$ respectively. The Milnor lattice $M$ of a hyperbolic singularity has a one-dimensional radical $\ker j$ generated by the vector $\delta_2 -\delta_1$. By \cite{Ga74b}, the monodromy group $\Gamma$ is the semi-direct product of the group $\ker j \otimes j(M)$ and the Coxeter group corresponding to the graph of Fig.~\ref{FigTpqr} with the vertex $\delta_2$ removed. It can also be described as the extended Weyl group of a generalized root system as defined by Looijenga \cite{Loo80}.

As an application of \cite{BWY19}, one obtains an extension of Corollary~\ref{cor:Deligne} to the hyperbolic singularities.

Looijenga (\cite{Loo80}, \cite[Chapter III.3]{Loo81}) gave a description of the complement of the discriminant of a simple elliptic or hyperbolic singularity as an orbit space $Y/\Gamma$. Using this, H.~van der Lek \cite{vdL83} gave a presentation of the fundamental group of the discriminant complement for such singularities generalizing the results for the simple singularities.

Besides these singularities, there are 14 exceptional unimodal singularities. Equations of these singularities can be found in \cite{Arn75,AGV85}. Coxeter-Dynkin diagrams for these singularities were also calculated by Gabrielov \cite{Ga74b}. He claimed that by  change-of-basis transformations $\alpha_i(j)$ and $\beta_i(j)$, the Coxeter-Dynkin diagrams can be reduced to the ``normal form'' depicted in Fig.~\ref{FigSpqr} for certain triples $(p,q,r)$.
\begin{figure}
$$
\xymatrix{ 
 & & & *{\bullet} \ar@{-}[d] \ar@{}^{\delta_{3}}[r] & & & \\
 & & & *{\bullet} \ar@{==}[d] \ar@{-}[dr]  \ar@{-}[ldd] \ar@{}^{\delta_{2}}[r]
 & & &  \\
 *{\bullet} \ar@{-}[r] \ar@{}_{\delta^2_1}[d]  & {\cdots} \ar@{-}[r]  & *{\bullet} \ar@{-}[r] \ar@{-}[ur]   \ar@{}_{\delta^2_{q-1}}[d] & *{\bullet} \ar@{-}[dl] \ar@{-}[r] \ar@{}^{\delta_{1}}[d] & *{\bullet} \ar@{-}[r]  \ar@{}^{\delta^3_{r-1}}[d]  & {\cdots} \ar@{-}[r]  &*{\bullet} \ar@{}^{\delta^3_1}[d]   \\
 & &   *{\bullet} \ar@{-}[dl] \ar@{}_{\delta^1_{p-1}}[r]  & & & & \\
 & {\cdots} \ar@{-}[dl] & & & & & \\
*{\bullet}  \ar@{}_{\delta^1_1}[r] & & & & & &
  } 
$$
\caption{The graph $S_{p,q,r}$} \label{FigSpqr}
\end{figure}
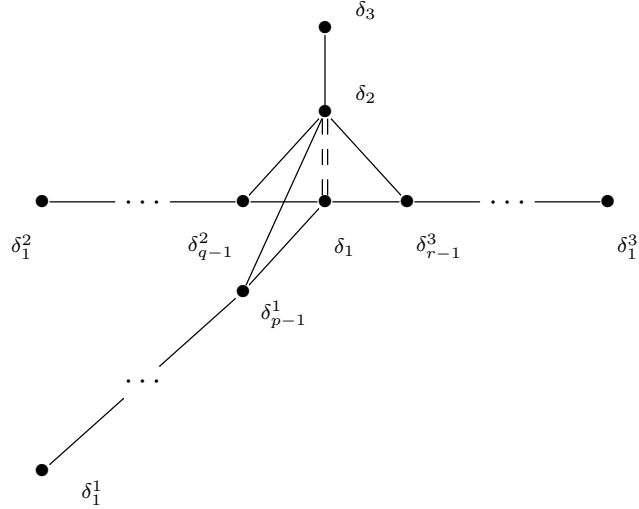
The author showed in his PhD thesis \cite{Eb80} that these diagrams are in fact Coxeter-Dynkin diagrams with respect to distinguished bases. The necessary braid group transformations are indicated in the appendix of the thesis which is not published in \cite{Eb81}. However, they can be found in the paper \cite{Eb96}. Arnold observed a strange duality between the 14 exceptional unimodal singularities \cite{Arn75}. The relation to homological mirror symmetry is explained in the survey article \cite{Eb17}. For a description of the monodromy groups see Sect.~\ref{sect:mono}.

V.~I.~Arnold also classified the bimodal singularities \cite{Arn75, AGV85}. Gabrie\-lov computed Coxeter-Dynkin diagrams with respect to distinguished bases for the singularities of all the series of Arnold, including the bimodal singularities \cite{Ga79}. The author suggested a ``normal form'' for the Coxe\-ter-Dynkin diagrams with respect to distinguished bases for the bimodal singularities \cite{Eb81, Eb83b}. Jointly with D.~Ploog \cite{EP13}, he gave a geometric construction of these diagrams. Moreover, he suggested a ``normal form'' for the Coxeter-Dynkin diagrams with respect to weakly distinguished bases of all the singularities of Arnold's series  and calculated their  intersection forms \cite{Eb81}.

Il'yuta \cite{Il87} formulated two conjectures relating the shape of Coxe\-ter-Dynkin diagrams to the modality of the singularity. He used the definition of the Coxeter-Dynkin diagram of \cite[2.8]{AGV88}: It is a graph with simple edges where the edge between $\delta_i$ and $\delta_j$ has the weight $\langle \delta_i, \delta_j \rangle$. (We assume $n \equiv 2 \ \mbox{mod} \, 4$.) A {\em monotone cycle}\index{monotone cycle}\index{cycle!monotone} in a Coxeter-Dynkin diagram is a sequence of vertices $(\delta_{i_1}, \ldots , \delta_{i_k})$ where $i_1 <  i_2 <  \ldots  < i_k$ and $\delta_{i_j}$ is connected to $\delta_{i_{j+1}}$ for $j=1, \ldots , k$ and $j+1$ taken modulo $k$. The {\em weight}\index{weight} of a monotone cycle is the product $\prod \langle \delta_{i_j}, \delta_{i_{j+1}} \rangle$ where the product is over $j=1, \ldots , k$ and $j+1$ taken modulo $k$. Now Il'yuta conjectured:

\begin{conjecture}[Il'yuta] \label{conj:I1}
The minimum over all $D \in \calD$ of the smallest number of edges that have to be deleted in order that $D$ does not contain monotone cycles is equal to the modality of the singularity.
\end{conjecture}

\begin{conjecture}[Il'yuta] \label{conj:I2}
The minimum over all $D \in \calD$ of the number of edges of $D$ of negative weight is equal to the modality of the singularity.
\end{conjecture}

Il'yuta showed that these conjectures hold for the unimodal singularities. The author \cite{Eb96} showed that both conjectures are even true for the unimodal singularities if one counts an edge of weight $\langle \delta_i, \delta_j \rangle$ as $|\langle \delta_i, \delta_j \rangle|$ edges as in the definition of the Coxeter-Dynkin diagram in Sect.~\ref{sect:CD}. However, he gave counterexamples to both conjectures for the bimodal singularities. Il'yuta also found other characterizations of Coxeter-Dynkin diagrams of the simple singularities \cite{Il94, Il95, Il97}.

Using the Lyashko-Looijenga mapping, M.~L\"onne \cite{Loe07, Loe10} gave a presentation of the fundamental group of the discriminant complement for Brieskorn-Pham singularities which is related to the intersection matrix with respect to a distinguished basis considered in Sect.~\ref{sect:comp}.

V.~A.~Vassiliev listed some problems about the Lyashko-Looijenga mapping for non-simple singularities in \cite{Va15}.

%%%%%%%%%%%%%%%%%%%%
\section{The monodromy group} \label{sect:mono}
In this section we give a description of the monodromy group in the general case.

Let $M$ be a lattice which is either symmetric and even or skew symmetric.
Let $\eps \in \{+1,-1\}$ and let $\Lambda$ be a subset of $M$. If $M$ is symmetric we demand that $\langle \delta , \delta \rangle = 2\eps$ for all $\delta \in \Lambda$. We define an automorphism $s_\delta \in {\rm Aut}(M)$ by 
\[
s_\delta(v):=v-\eps \langle v, \delta \rangle \delta
\]
for all $v \in M$. Then $s_\delta$ is a reflection in the symmetric case and a symplectic transvection in the skew symmetric case. Let $\Gamma_\Lambda \subset {\rm Aut}(M)$ be the subgroup of ${\rm Aut}(M)$ generated by the transformations $s_\delta$, $\delta \in \Lambda$.

\begin{definition} The pair $(M,\Lambda)$ is called a {\em vanishing lattice}\index{vanishing lattice}\index{lattice!vanishing}, if it satisfies the following conditions:
\begin{itemize}
\item[(i)] $\Lambda$ generates $M$.
\item[(ii)] $\Lambda$ is an orbit of $\Gamma_\Lambda$ in $M$.
\item[(iii)] If ${\rm rank} \, M >1$, then there exist $\delta, \delta' \in \Lambda$ such that $\langle \delta, \delta' \rangle =1$.
\end{itemize}
\end{definition}

It follows from Corollary~\ref{cor:vanishing} that, if it is not true that both $n$ is odd and 0 is a non-degenerate critical point of $f$, then the pair $(\widetilde{H}_n(X_s), \Lambda^\ast)$ is a vanishing lattice with $\eps=(-1)^{n(n-1)/2}$ and $\Gamma$ is the corresponding monodromy group.

We introduce some more algebraic notions. Let 
 $M^\#=\on{Hom}(M,\ZZ)$ be the dual module and $j: M \to M^\#$ be 
 the canonical homomorphism.
 A homomorphism $h:M\to M$ induces a homomorphism $h^t:M^\#\to M^\#$ of the dual modules.
 If $h$ leaves the bilinear form $\langle \, ,\, \rangle$ invariant, then $h^t(j(M))\subset j(M)$.
 An automorphism $h\in\on{Aut}(M)$ thus induces a homomorphism $h^t:M^\#/j(M)\to M^\#/j(M)$.
 Let $\on{Aut}^\#(M)\subset\on{Aut}(M)$ be the subgroup of those automorphisms $h\in\on{Aut}(M)$ with
 $h^t=\on{id}_{M^\#/j(M)}$.
 
 %%%% skew symmetric case
Now let $M$ be a skew symmetric  lattice.
 It has a basis
 \[
 (e_1,f_1,\ldots,e_m,f_m,g_1,\ldots,g_k)
 \]
 such that 
 \[
 \langle e_i, f_i \rangle = - \langle f_i, e_i \rangle = d_i \ \mbox{for } d_i \in \ZZ,\  d_i \geq 1,\  i=1, \ldots m,
 \] 
 all other inner products are equal to zero, 
 and $d_{i+1}$ is divisible by $d_i$ for $i=1,\ldots,m-1$.
 Such a basis is called a {\em symplectic}\index{symplectic basis}\index{basis!symplectic} basis.

 Let $(e_1,f_1,\ldots,e_m,f_m,g_1,\ldots,g_k)$ be a symplectic basis of $M$.
 Let $\eta_2$ be the exponent of $2$ in the prime factor decomposition of $d_m$.
 Let $\mu=2m+k$.
 We identify $M$ with $\ZZ^\mu$ through the symplectic basis $(e_1,f_1,\ldots,e_m,f_m,g_1,\ldots,g_k)$.
 A subgroup $G\subset\on{Aut}^\#(M)$ then corresponds to a subgroup
 $\rho(G)\subset\on{Sp}^\#(\mu,\ZZ)$, where $\on{Sp}^\#(\mu,\ZZ)$ is the corresponding subgroup of the symplectic group
 \[
 \on{Sp}(\mu,\ZZ)= \{ A \in \on{GL}(\mu, \ZZ) \, | \, A^t J A = J \}.
 \]
 Let $r\in\NN\setminus\{0\}$.
 A subgroup $G\subset\on{Aut}^\#(M)$ is called a {\em congruence subgroup modulo}\index{congruence subgroup} $r$
 if
  $$
  \rho(G)=\{A\in\on{Sp}^\#(\mu,\ZZ)\ |\ A\equiv E\on{mod} r\}.
  $$
 Here $E$ is the unit matrix and $A\equiv E\on{mod}r$ means that
 $a_{ij}\equiv\delta_{ij}\on{mod}r$ for all $1\leq i,j\leq\mu$, where $A=(a_{ij})$.

 A congruence subgroup is obviously of finite index in the group
 $\on{Aut}^\#(M)=\on{Sp}^\#(M)$.
 
The following theorem was proved by W.~A.~M.~Janssen \cite{Jan83}  based on
 previous work of A'Campo \cite{A'C79}, B.~Wajnryb
 \cite{Wa80}, and S.~V.~Chmutov \cite{Chm82, Chm83}. 
The notation $\langle v,M\rangle=\ZZ$ means that there is a $y\in M$ with $\langle v,y\rangle=1$. 
We write $a\in\Lambda\on{mod}2$ if there is an element $b\in M$ with $a-2b\in\Lambda$.

 \begin{theorem}[Janssen] \label{thm:Janssen}
 Let $(M, \Lambda)$ be a skew symmetric vanishing lattice. Then
\begin{itemize}
\item[{\rm (i)}]  $\Gamma_\Lambda$ contains the congruence subgroup modulo
 $2^{\eta_2+1}$ of $\on{Sp}^\#(M)$,
\item[{\rm (ii)}] $\Lambda=\{v\in M\ |\ \langle v,M\rangle=\ZZ\text{ and }
 v\in\Lambda\on{mod}2\}$.
 \end{itemize}
 \end{theorem}
 
As a corollary, we get the following result:
\begin{corollary}
Let $f:(\CC^{n+1},0)\to(\CC,0)$ be a holomorphic function germ with an isolated singularity at $0$ and let $n$ be odd. Then
\begin{itemize}
\item[{\rm (i)}]  $\Gamma$ contains the congruence subgroup modulo
 $2^{\eta_2+1}$ of $\on{Sp}^\#(M)$,
\item[{\rm (ii)}] $\Lambda^*=\{v\in M\ |\ \langle v,M\rangle=\ZZ\text{ and }
 v\in\Lambda^*\on{mod}2\}$.
 \end{itemize}
In $\on{(ii)}$ it is assumed that $0$ is not a non-degenerate critical point of $f$. 
 \end{corollary}

Janssen also proved a version of Theorem~\ref{thm:Janssen} for skew symmetric vanishing lattices over the field $\FF_2$ and classified skew symmetric vanishing lattices over $\FF_2$ \cite{Jan83} and over $\ZZ$ \cite{Jan85}. B.~Shapiro, M.~Shapiro,  and A.~Vainshtein \cite{SSV98} applied these results to certain enumeration problems.

%%%% symmetric case
 Now let $M$ be symmetric and let $\eps\in\{-1,+1\}$. We put
 \[
  \overline{M} :=M/\on{ker}j, \quad 
  \overline{M}_\RR:=\overline{M} \otimes\RR.
 \]
 Then $\overline{M}_\RR$ is a finite-dimensional real vector space with a nondegenerate symmetric bilinear form.
 Let $h\in\on{Aut}(M)$ and $\bar h$ the induced element in $O(\overline{M}_\RR)$.
 The transformation $\bar h$ can be written as a product of
 reflections
  $$
  \bar h=s_{v_1}\circ\ldots\circ s_{v_r}
  $$
 with $v_i\in \overline{M}_\RR$, $\langle v_i,v_i\rangle\neq0$, $i=1,\ldots,r$. We define
  $$
  \nu_\eps(h):=\left\{\begin{array}{cl} +1 & \text{ if }\eps\langle
  v_i,v_i\rangle<0\text{ for an even number of indices,}\\
  -1 & \text{ otherwise.}\end{array}\right.
  $$
 The homomorphism $\nu_\eps:\on{Aut}(M)\to \{-1,+1\}$ is called the {\em real $\eps$-spinor norm}.\index{spinor norm}

 \begin{definition}
 We define a subgroup $O_\eps^*(M)\subset O(M)$\index{$O_\eps^*(M)$} as follows:
  $$
  O_\eps^*(M):=\on{Aut}^\#(M)\cap\on{ker}\nu_\eps.
  $$
  \end{definition}

 If $M$ is non-degenerate, then $M^\#/j(M)$ is a finite group and hence
 $O^\#(M)=\on{Aut}^\#(M)$ is a subgroup of finite index in $O(M)=\on{Aut}(M)$.
 The subgroup $\on{ker}\nu_\eps\subset O(M)$ is of index $\leq2$ in $O(M)$.
 Thus if $M$ is non-degenerate, then $O_\eps^*(M)$ is a subgroup of finite index in $O(M)$.

Using a result of M.~Kneser \cite{Kn81}, the author proved the following theorem \cite{Eb84}. A unimodular hyperbolic plane is a two-dimensional lattice with the bilinear form given by
\[
\begin{pmatrix} 0 & 1\\ 1 & 0
\end{pmatrix} .
\]
\begin{theorem} \label{thm:symm-van}
Let $(M, \Lambda)$ be an even symmetric vanishing lattice. Assume that $M$ contains a six-dimensional sublattice $K \subset M$ which is the orthogonal direct sum of two unimodular hyperbolic planes and a lattice of type $\eps A_2$. Assume moreover $\{ v \in K \, | \, \langle v, v \rangle=2 \eps \} \subset \Lambda$. Then
\begin{itemize}
\item[{\rm (i)}] $\Gamma_\Lambda=O_\eps^*(M) ,$
\item[{\rm (ii)}] $\Lambda=\{v\in M\ |\ \langle v,v\rangle=2\eps\text{ and } \langle
 v,M\rangle=\ZZ\} .$
\end{itemize}
\end{theorem}

From Theorem~\ref{thm:symm-van} one can derive the following theorem \cite{Eb84}. The statement for the exceptional unimodal singularities was already proven by  H.~Pinkham \cite{Pi77} (see also \cite{Eb83a} for the history of the problem and previous results).
It was noticed by Looijenga that (ii) is a consequence of (i).

 \begin{theorem}\label{thm:arithmetic}
 Let $f:(\CC^{n+1},0)\to(\CC,0)$ be a holomorphic function germ with an isolated singularity at $0$ and let $n$ be even,
 $\eps=(-1)^{n(n-1)/2}$.
 Suppose that $f$ is not of type $T_{p,q,r}$ with $\frac{1}{p}+\frac{1}{q}+\frac{1}{r} < 1$ and $(p,q,r)\neq
 (2,3,7),(2,4,5),(3,3,4)$.
Then
\begin{itemize}
\item[{\rm (i)}] $\Gamma=O_\eps^*(M) ,$
\item[{\rm (ii)}] $\Lambda^*=\{v\in M\ |\ \langle v,v\rangle=2\eps\text{ and } \langle
 v,M\rangle=\ZZ\} .$
\end{itemize}
In $\on{(ii)}$ it is assumed that $0$ is not a non-degenerate critical point of $f$. 
 \end{theorem}
 
 \begin{remark} \label{rem:hyperbolic}
Theorem~\ref{thm:arithmetic} follows for the simple and simple elliptic singularities by the results stated in Sect.~\ref{sect:special}. It is false for the singularities of type $T_{p,q,r}$ with $\frac{1}{p}+\frac{1}{q}+\frac{1}{r} < 1$ and $(p,q,r)\neq 
  (2,3,7),(2,4,5),(3,3,4)$, see \cite[\S 3]{Eb81}. This follows from the fact that the graph of Fig.~\ref{FigTpqr} with the vertex $\delta_2$ removed and with $\frac{1}{p}+\frac{1}{q}+\frac{1}{r} < 1$ defines a Coxeter system of hyperbolic type if and only if $(p,q,r)= (2,3,7), (2,4,5), (3,3,4)$ \cite[Ch.~V, \S~4, Exercice 12]{Bou02}. The three singularities $T_{p,q,r}$ with these values of $(p,q,r)$ are the {\em minimal hyperbolic singularities}\index{minimal hyperbolic singularities}\index{hyperbolic singularities!minimal}\index{singularities!minimal hyperbolic}. Theorem~\ref{thm:arithmetic} was proved for these singularities by Brieskorn (\cite[Theorem~2]{Br81}, but no proof is given). A proof following Brieskorn's proof can be found in \cite[5.5]{Eb87}. A'Campo (unpublished) and Looijenga showed that the monodromy groups of these singularities have exponential growth. Looijenga's proof is published in \cite[Appendix~II]{Du79}.
 \end{remark}
 
%Theorem~\ref{thm:arithmetic} also has an analogue for complete
%intersections \cite{Eb87}.

%%%%%%%%%%%%%%%%%%%%
\section{Topological equivalence} \label{sect:top}
We shall now consider the question to which extent the invariants determine the topological type of the singularity.

The topological type of a singularity $f: (\CC^{n+1},0) \to (\CC,0)$ is described by the (local) embedding of the variety $f^{-1}(0)$ in a neighborhood of the singular point $0 \in \CC^{n+1}$.

\begin{definition} Two singularities $f,g: (\CC^{n+1},0) \to (\CC,0)$ are {\em topologically equivalent}\index{topologically equivalent}\index{equivalent!topologically} if there is a homeomorphism of neighborhoods $U$ and $V$ of the origin which maps  $f^{-1}(0) \cap U$ to $g^{-1} \cap V$.
\end{definition} 

By \cite[Theorem~2.10]{Mi68}, the variety $f^{-1}(0)$ is locally the cone over its link $K$. By \cite{Mi68}, the link is a fibered knot.
A.~Durfee \cite{Du74} proved the following theorem.

\begin{theorem}[Durfee] \label{thm:Durfee}
Let $n \geq 3$. There is a one-to-one correspondence of isotopy classes of fibered knots in $S^{2n+1}$ and equivalence classes of integral unimodular bilinear forms given by associating to each fibered knot its Seifert form.
\end{theorem}

In view of the preceding remarks and Theorem~\ref{thm:SLH} we obtain the following corollary.

\begin{corollary} \label{cor:D^*}
For $n \neq 2$ the set $\calD$ of $f: (\CC^{n+1},0) \to (\CC,0)$ determines $f$ up to topological equivalence.
\end{corollary}

This corollary was also proved by S.~Szczepanski \cite{Sz87}. Moreover, she showed in \cite{Sz89} the following theorem.

\begin{theorem}[Szczepanski] \label{thm:Sz}
Two singularities $f,g: (\CC^3,0) \to (\CC,0)$ are topologically equivalent if
\begin{itemize}
\item[{\rm (i)}] the singularities have a common Coxeter-Dynkin diagram with respect to distinguished bases, and
\item[{\rm (ii)}] the Milnor fibers have homeomorphic boundaries and the algebraic isomorphism of the Milnor lattices induced by the common Coxeter-Dynkin diagram is realized geometrically by either an inclusion of one Milnor fiber into the other or a homotopy equivalence of the Milnor fibers which induces a homeomorphism of the boundaries.
\end{itemize}
\end{theorem}

There is also the notion of $\mu$-homotopy or $\mu$-equivalence (see \cite{Br88}).

\begin{definition} Two singularities $f_0,f_1: (\CC^{n+1},0) \to (\CC,0)$ are {\em $\mu$-equiva\-lent}\index{$\mu$-equivalent} if there is a family $f_t : (\CC^{n+1},0) \to (\CC,0)$ of analytic function germs with isolated singularities at the origin continuously depending on the parameter $t \in [0,1]$ with constant Milnor number $\mu(f_t)$.
\end{definition}

L\^e D\~ung Tr\'ang and C.~P.~Ramanujam \cite{LR76} proved the following theorem.

\begin{theorem}[L\^e-Ramanujam] If $n \neq 2$, then $\mu$-equivalent singularities are topologically equivalent.
\end{theorem}

The following proposition was proved by Gabrielov \cite[Proposition~1]{Ga74b}.
\begin{proposition}[Gabrielov] \label{prop:Gab}
For two $\mu$-equivalent singularities there exist distinguished bases whose Coxeter-Dynkin diagrams coincide.
\end{proposition}

Using this proposition, one obtains the L\^e-Ramanujam theorem as a consequence of Corollary~\ref{cor:D^*}. Moreover, one can derive from Theorem~\ref{thm:Sz} a L\^e-Ramanujam theorem for $n=2$, see \cite{Sz89}.

Now let $f:(\CC^{n+1},0)\to(\CC,0)$ be an isolated singularity satisfying the conditions of Theorem~\ref{thm:arithmetic}. It follows from Theorem~\ref{thm:arithmetic} that the invariants $\Gamma$ and $\Lambda^*$ are completely determined by $M$. The author \cite{Eb81} has found examples of pairs of singularities (e.g.\ the bimodal singularities $Z_{17}$ and $Q_{17}$ in Arnold's notation) which have the same Coxeter-Dynkin diagrams with respect to weakly distinguished bases and the invariants $M$, $\Gamma$, and $\Lambda^*$ are the same, but the invariants $\calB^*$ and $\calD^*$ are different, the classical monodromy operators are not conjugate to each other,  and the singularities are not topologically equivalent.

We conclude the article with some open questions which were posed to the author by late Brieskorn (around 1982?). We keep the condition that $f:(\CC^{n+1},0)\to(\CC,0)$ is a singularity satisfying the conditions of Theorem~\ref{thm:arithmetic}. Let $n \equiv 2 \ \mbox{mod} \, 4$ and let $\mu$ be the Milnor number of $f$.

\begin{problem}[Brieskorn] Let $M$ be the Milnor lattice (of rank $\mu$) and $\Gamma$ be the monodromy group of $f$. Let 
\[
\Lambda := \{ v \in M \, | \, \langle v , v \rangle = -2 \}.
\]
Then $\Gamma$ acts on $\Lambda$. Are there only finitely many orbits?
\end{problem}

\begin{problem}[Brieskorn] 
Let 
\[
\calB_0 := \{ (e_1, \ldots , e_\mu) \in \Lambda^\mu \, | \, \langle e_1, \ldots , e_\mu \rangle_\ZZ=M \}.
\]
Then the group ${\rm Br}_\mu^\rtimes$ acts on $\calB_0$. Are there only finitely many orbits? Alternatively, one can consider the set 
\[
\widetilde{\calB}_0:= \{ (e_1, \ldots , e_\mu) \in (\Lambda^*)^\mu \, | \, \langle e_1, \ldots , e_\mu \rangle_\ZZ=M \}.
\]
\end{problem}

\begin{problem}[Brieskorn] \label{prob:3}
The group $\Gamma$ acts on $\calB_0$ (or $\widetilde{\calB}_0$) by 
\[
\gamma(e_1, \ldots , e_\mu)=(\gamma e_1, \ldots , \gamma e_\mu).
\]
This action commutes with the action of ${\rm Br}_\mu^\rtimes$. Are there only finitely many $\Gamma$-equivalence classes of ${\rm Br}_\mu^\rtimes$-orbits?
\end{problem}
Very little is known about these problems. The answers to these questions are trivially yes for the simple singularities, since the sets $\Lambda$, $\calB_0$, and $\widetilde{\calB}_0$ are finite in this case. We have $\Lambda=\Lambda^*$ (and hence there is only one $\Gamma$-orbit) for the simple, simple elliptic, and minimal hyperbolic singularities (for the latter ones see \cite[Proposition~5.5.1]{Eb87}). 

An element $c \in \Gamma$ for which there exists a basis $(e_1, \ldots, e_\mu) \in \calB_0$ such that $c=s_{e_1} \cdots s_{e_\mu}$ is called a {\em quasi Coxeter element}\index{quasi Coxeter element}\index{Coxeter element!quasi}.

\begin{problem}[Brieskorn] \label{prob:4}
Let $c \in \Gamma$ be a quasi Coxeter element and let
\[ \calB_{0,c} := \{ (e_1, \ldots , e_\mu) \in \calB_0 \, | \, s_{e_1} \cdots s_{e_\mu}=c \}.
\]
The set  $\calB_{0,c}$ is invariant under the action of the group ${\rm Br}_\mu^\rtimes$. What is the relation between the orbits of ${\rm Br}_\mu^\rtimes$ on $\calB_0$ and the sets $\calB_{0,c}$?
\end{problem}

For the simple singularities, the quasi Coxeter elements were determined up to conjugacy by E.~Voigt \cite{Vo85a, Vo85b} and he showed that the group ${\rm Br}_\mu^\rtimes$ acts transitively on $\calB_{0,c}$ for each quasi Coxeter element $c$. For $c$ being the classical monodromy operator, it is known for the simple (Corollary~\ref{cor:Deligne}), the simple elliptic \cite{Kl86}, and the hyperbolic singularities \cite{BWY19} that the group ${\rm Br}_\mu^\rtimes$ acts transitively on $\calB_{0,c}$ (see Sect.~\ref{sect:special} above). To the author's knowledge, this is all what is known about Problem~\ref{prob:4}.

%%%%%%%%%%%%%%%%%%%%%%%%
%\section*{References}

\begin{sloppypar}

\end{sloppypar}

\printindex

\end{document}